\documentclass{aastex701}

\usepackage{defs,fonts}
\usepackage{amsfonts,amsmath,amssymb,amsthm,color}
\usepackage{caption,subcaption}
\usepackage{algorithm}
\usepackage{algpseudocode}

\usepackage{enumitem}
\newlist{steps}{enumerate}{1}
\setlist[steps, 1]{leftmargin = \leftmargin+1cm, rightmargin = \rightmargin+1cm, label = Step \arabic*:}

\usepackage{mathtools}

\DeclareMathOperator*{\argmin}{arg\,min}

\newcommand{\innerproduct}[2]{\Big\langle #1 \Big\rangle_{#2}}

\newtheorem{dfn}{Definition}

\newtheorem{rem}{Remark}
\newtheorem{thm}{Theorem}
\newtheorem{lem}{Lemma}

\usepackage{hyperref}
\hypersetup{
    colorlinks=true,
    linkcolor=blue,
    filecolor=magenta,      
    urlcolor=cyan,
    pdftitle={Overleaf Example},
    pdfpagemode=FullScreen,
    citecolor = blue,
    }

\begin{document}

\title{Mesh-Based Filtering to Alleviate Time-Step Restrictions in Runge--Kutta Discontinuous Galerkin Methods in Spherical-polar Coordinates: Application to the Euler Equations
\footnote{This manuscript has been authored in part by UT-Battelle, LLC, under contract DE-AC05-00OR22725 with the US Department of Energy (DOE). 
The US government retains and the publisher, by accepting the article for publication, acknowledges that the US government retains a nonexclusive, paid-up, irrevocable, worldwide license to publish or reproduce the published form of this manuscript, or allow others to do so, for US government purposes. 
DOE will provide public access to these results of federally sponsored research in accordance with the DOE Public Access Plan (http://energy.gov/downloads/doe-public-access-plan).}
}

\author[orcid=0009-0004-7188-1172,gname=Joseph,sname=Hunter]{Joseph Hunter}
\affiliation{Department of Mathematics, Ohio State University, Columbus, OH, 43210, USA}
\email[show]{hunter.926@buckeyemail.osu.edu}  

\author[orcid=0000-0003-1251-9507,gname=Eirik,sname=Endeve]{Eirik Endeve}
\affiliation{Computer Science and Mathematics Division, Oak Ridge National Laboratory, Oak Ridge, TN 37831, USA}
\affiliation{Department of Physics and Astronomy, University of Tennessee Knoxville, Knoxville, TN 37996, USA}
\email{endevee@ornl.gov}

\author[orcid=0000-0002-3504-6194,gname=Yulong,sname=Xing]{Yulong Xing}
\affiliation{Department of Mathematics, Ohio State University, Columbus, OH, 43210, USA}
\email{xing.205@osu.edu}

\begin{abstract}
    We propose a mesh-based filtering approach to alleviate the severe timestep restrictions arising in explicit Runge--Kutta discontinuous Galerkin (RKDG) methods formulated in spherical-polar coordinates.  
    The filter enables stable evolution on the original logically Cartesian mesh while using larger time steps associated with an auxiliary merged mesh constructed to eliminate the extreme cell anisotropies produced by converging coordinate lines near coordinate singularities.
    The filter is implemented as a sequence of post-processing operations applied within an $s$-stage RK time integrator, making it straightforward to incorporate into existing structured-mesh DG frameworks. 
    We analyze the filter in one spatial dimension and prove that the filtered RKDG method is equivalent to evolving the RKDG discretization on a nonuniform mesh obtained by merging selected elements of the underlying uniform mesh.  
    This equivalence implies that the filtered method inherits the accuracy and stability properties of the corresponding RKDG discretization on the merged mesh.  
    We apply the mesh-based filter to an existing RKDG method for the Euler equations in spherical-polar coordinates and demonstrate, through selected two- and three-dimensional examples, its effectiveness in accelerating simulations through significantly larger stable timesteps.
\end{abstract}


\section{Introduction}
\label{sec:intro}

In this paper we present a mesh-based filtering technique for alleviating the severe time step restrictions that arise when discontinuous Galerkin (DG) methods are applied on structured spherical-polar meshes with explicit time integration.  
Although our primary application is the nonrelativistic Euler equations, the filtering procedure is independent of the underlying system of partial differential equations (PDEs).  
The filter is implemented as a post-processing operation within a multi-stage time integrator, making it straightforward to incorporate into existing structured-mesh DG frameworks.  
Conceptually, the filter enables the solution to be evolved on the original spherical-polar mesh while employing time steps characteristic of an auxiliary mesh with reduced cell anisotropy.  
This is particularly advantageous for simulations of spherical objects, where converging coordinate lines produce highly anisotropic cells near the origin and poles that severely restrict the allowable time step.  
By mitigating these geometric time step constraints, the mesh-based filter can substantially reduce the computational cost of DG simulations in spherical-polar coordinates.  
The present work focuses on the numerical properties of the filtering strategy and its effectiveness in alleviating time step restrictions.  
Questions related to load balancing, communication in distributed parallel implementations, and scalability are left for future work.  

Spherical-polar coordinates $(r,\theta,\phi)$, where $r$ is radius, $\theta$ is polar angle, and $\phi$ is azimuthal angle, are widely used for simulations of nearly spherical objects such as stars, which are gravitationally stratified and often remain close to spherical equilibrium configurations \citep[e.g.,][]{chandrasekhar_1967,shapiroTeukolsky_1983}.  
Their principal advantage is that they naturally align with both the geometry and radial stratification of such objects.  
In particular, nonuniform radial spacing can efficiently resolve the multiscale structure of stellar interiors and atmospheres, often through logarithmic spacing at large radii \citep{Muller2020}.  
Furthermore, spherical or axial symmetry can be imposed to reduce the problem to one or two spatial dimensions, enabling computationally inexpensive studies that build intuition before undertaking fully three-dimensional simulations.
Cartesian coordinates $(x,y,z)$, by contrast, are straightforward to implement, avoid coordinate singularities, and yield uniform mesh aspect ratios.  
However, when applied to spherical objects they may introduce grid-aligned departures from spherical symmetry and can struggle to efficiently represent strongly stratified multiscale flows on a single-resolution mesh \citep{Muller2020}.  
Various techniques, including adaptive mesh refinement \citep{BERGER198964,Fryxell_2000,Ott_2012} and hybrid Cartesian--spherical mesh constructions \citep{Scheidegger_2010}, have been employed to mitigate these difficulties.  
Nevertheless, for many problems involving nearly spherical objects, spherical-polar coordinates remain a natural and efficient choice.  

While spherical-polar coordinates are a natural choice for simulating spherical objects, they do introduce two significant challenges.
The first challenge is the presence of coordinate singularities at the origin and along the poles \citep[e.g.,][]{monchmeyerMuller_1989,blondinLufkin_1993}.
Considering the line element
$
	ds^2 = dr^2 + r^2d\theta^2+r^2\sin^2\theta d\phi^2,
$
the proper lengths in the angular directions, $dl_{\theta}=rd\theta$ and $dl_{\phi}=r\sin\theta d\phi$, vanish as $r \to 0$, while $dl_{\phi}\to0$ as $\theta\to0$ and $\theta\to\pi$, implying that coordinate basis vectors are not uniquely defined at the origin and poles.  
Furthermore, differential operators may contain geometric factors such as $1/r^2$ and $1/\sin\theta$, whose cancellation relies on the regularity of the solution, and ensuring these cancellations at the discrete level can be challenging.  
Various meshes, including Yin--Yang meshes \citep{Kageyama_2004,Wongwathanarat_2010}, cubed-sphere meshes \citep{ronchi_etal_1996}, geodesic meshes \citep{florinski_etal_2013}, and hybrid Cartesian--cubed-sphere meshes \citep[e.g.,][]{ott_etal_2013} have been developed to mitigate these difficulties.
An alternative is the reference metric-approach \citep[see, e.g.,][and references therein]{Baumgarte_2013,Mewes_2020}, which recasts coordinate-dependent geometric terms using a prescribed reference metric, facilitating the extension of Cartesian formulations to spherical-polar coordinates, although maintaining strict conservation may require additional care.
Addressing discretization issues associated with coordinate singularities is not the primary focus of the present work.  

The second challenge, and the focus of this paper, is the severe time step restriction imposed by the Courant--Friedrichs--Lewy (CFL) condition for explicit time integration \citep{Courant_1928}.  
Because the proper angular cell widths are $\Delta l_{\theta}=r\Delta\theta$ and $\Delta l_{\phi}=r\sin\theta\Delta\phi$, the smallest cell dimensions occur near the origin and poles where coordinate lines converge.  
For a structured spherical-polar mesh, the maximum stable time step is therefore controlled by these shrinking proper lengths rather than by the radial spacing.  
Near the origin in two dimensions (2D), where $r\sim\Delta r$, the time step scales as $\Delta t\sim\Delta r \Delta \theta$.  
In three dimensions (3D), the additional factor $\Delta l_{\phi}=r\sin\theta\Delta\phi$ further restricts the time step near the poles, yielding $\Delta t\sim\Delta r \Delta\theta \Delta\phi$.
In contrast, the time step on a Cartesian mesh scales with the minimum Cartesian mesh spacing; $\Delta t\sim\min(\Delta x,\Delta y,\Delta z)$.  
Thus, the CFL restriction is typically not a consequence of physical resolution requirements, but rather of geometric over-resolution introduced by the convergence of coordinate lines near coordinate singularities, which produces highly anisotropic cells near the origin and poles and leads to excessively small time steps and increased computational cost. 
It is this geometric CFL restriction that the mesh-based filtering strategy developed in this work seeks to alleviate.  

Several approaches have been proposed to mitigate the severe CFL restrictions associated with spherical-polar meshes.  
One strategy is to evolve the governing equations directly on an unstructured, or ``dendritic,'' mesh, which is designed to avoid the extreme cell anisotropies produced by converging coordinate lines \citep{Skinner2019}.  
While effective, such approaches introduce additional discretization complexity associated with mesh-resolution transitions.  
Another strategy is Fourier filtering, in which high-frequency angular modes are selectively damped to stabilize time steps that exceed the nominal CFL limit.  
Fourier filtering has been applied in the $\phi$-direction \citep{Muller2019} and in both the $\theta$- and $\phi$-directions \citep{ji2023}.  
However, because Fourier filtering modifies the solution through a global spectral representation, it may introduce Gibbs oscillations in the presence of shocks.  

Of particular relevance to the present work are mesh-coarsening or cell-merging approaches developed in the finite-volume (FV) community \citep[e.g.,][]{Asaithambi2017,Muller2015,Nakamura2019,cardall_etal_2026}.  
These methods retain the original structured spherical-polar mesh while employing an auxiliary coarsened mesh to determine the allowable time step, thereby avoiding many of the implementation challenges associated with fully unstructured discretizations.  
With an appropriate choice of auxiliary mesh, the time step can scale as $\Delta t\sim\Delta r$, making the CFL constraint comparable to that of a Cartesian mesh.  
Conceptually, the procedure consists of restricting the solution from the original mesh to the auxiliary mesh and subsequently prolonging the solution back to the original mesh, while evolving on the fine mesh with a time step determined by the coarsened mesh.  
The mesh-based filtering strategy proposed in this work is directly inspired by these cell-merging approaches, but is developed within the context of high-order DG methods.  
As discussed by \citet{Muller2019}, maintaining both conservation and monotonicity during the prolongation step can be challenging in FV methods.  

In this paper, we develop a mesh-based filtering approach for Runge--Kutta DG (RKDG) methods \citep{Cockburn_1998,cockburnShu_2001}, inspired by cell-merging techniques from the FV literature, to alleviate the severe CFL restrictions associated with spherical-polar meshes. 
DG methods combine features of both spectral and FV methods.  
While FV methods evolve the cell average of the solution and reconstruct local polynomial representations from neighboring cell averages, DG methods evolve an element-local polynomial representation of the solution.  
Equivalently, FV methods evolve the zeroth moment of the solution, whereas DG methods evolve both the cell average and a finite number of higher-order moments.  
This fundamental difference necessitates a framework that extends the cell-merging idea to accommodate the restriction and prolongation of higher-order polynomial representations between the fine and auxiliary meshes.  
When constant basis functions are employed, the DG method reduces to an FV discretization and existing cell-merging approaches apply directly.  
For higher-order polynomial representations, however, a more general framework is required to achieve the same alleviation of the time step restriction.
We refer to this generalized DG formulation of the cell-merging concept as mesh-based filtering.  

The auxiliary meshes considered in this work are designed so that the proper angular lengths $\Delta l_{\theta}$ and $\Delta l_{\phi}$ remain comparable to the radial spacing $\Delta r$, with no merging  applied where $\Delta l_{\theta}$ and $\Delta l_{\phi}$ exceed $\Delta r$.  
Consequently, the mesh-based filter primarily targets geometric over-resolution rather than physically required resolution.  
We do not expect significant degradation of spatial accuracy.  
At worst, the filtered solution is expected to exhibit the accuracy of the corresponding auxiliary merged mesh, but the mesh-based filter can be easily adjusted to satisfy the angular resolution requirements of a given application.  

We focus on the DG method for several reasons.  
Unlike continuous finite element and spectral methods, DG methods do not enforce continuity of the solution across element boundaries, and they are therefore, like FV methods, suitable for capturing shocks and other discontinuities.  
DG methods also support $hp$-adaptivity \citep{Remacle_2003}, enabling mesh refinement and variation of the local polynomial degree in response to solution features.  
Furthermore, they are naturally applicable to problems in curvilinear coordinates \citep{TEUKOLSKY2016333} and complex geometries \citep{gulizzi_etal_2022}.  
By evolving the polynomial representation directly, DG methods eliminate the need for reconstruction and achieve high-order accuracy on compact stencils, with communication limited to nearest neighbors, independent of the polynomial degree.  
At higher order, they trade increased arithmetic intensity for reduced communication and memory-access costs \citep[e.g.,][]{millerSchnetter_2017}.  
We refer to \citet{shu_2016,gassnerWinters_2021} and references therein for overviews on recent developments of DG methods.  
More recently, DG methods have gained popularity in computational astrophysics \citep[e.g.,][]{Radice_2011,Schall_2015,KIDDER201784,Fambri_2018,Pochik_2021,dunham_etal_2024,hunter_etal_2025,endeve_etal_2026}, motivating the development of techniques such as mesh-based filtering to improve the efficiency of high-resolution DG-based astrophysics simulations on spherical-polar meshes.  

The paper is organized as follows.  
In Section~\ref{sec:1DCellMerging}, we introduce and analyze the mesh-based filtering approach in an idealized one-dimensional setting, establish theoretical results concerning its stability, and verify these results through one linear and one nonlinear numerical test.  
Section~\ref{sec:2D3Dmerging} extends the formulation to spherical-polar coordinates in two and three spatial dimensions.  
In Section~\ref{sec:Results}, we apply the method to the nonrelativistic Euler equations in two- and three-dimensional numerical examples demonstrating that accurate solutions can be obtained with significantly larger time steps.  
Finally, in Section~\ref{sec:Conclusion} we summarize our findings, draw conclusions, and discuss directions for future work.  

\section{Mesh-Based Filtering in One Dimension: Formulation, Analysis, and Validation}
\label{sec:1DCellMerging}

In this section we formulate the mesh-based filtering procedure in one
spatial dimension and analyze its relation to a DG discretization on an auxiliary mesh where some cells are merged.
We refer to this as a merged mesh.
The one-dimensional setting is used to present the transfer
operation between the fine and merged meshes and to carefully study and validate its effect. 

We consider the scalar conservation law
\begin{equation}
\label{eq:linearadvectionPDE}
    \p_{t}u(t,x) + \p_{x}f(u(t,x)) = 0,
    \quad
    t\in(0,T],
    \quad
    x\in\Omega,
\end{equation}
on the domain $\Omega = [a,b]$, with periodic boundary conditions and initial condition $u_0 = u(0,x)$.
Here, $u(t,x)$ is the time-dependent exact solution, and $f$ is a flux function.
We choose Equation~\eqref{eq:linearadvectionPDE} as our model for its simplicity, however the mesh-based filtering process we will describe in Sections \ref{sec:1DCellMerging} and \ref{sec:2D3Dmerging} can be applied to a broad class of models.

\subsection{Notation and Standard DG Method}
\label{subsec:StandardDG}
We partition the domain $\Omega = [a,b]$ into $N$ non-overlapping elements $I_j = [x_{j-1/2},x_{j+1/2}]$ for $j=1,\ldots,N$, such that $x_{1/2} = a$ and $x_{N+1/2} = b$.
We denote the length of an element by $|I_j| = x_{j+1/2}-x_{j-1/2}$, and write
\begin{equation}
    \cT=\{I_j\}_{j=1}^N,
    \qquad
    \overline\Omega=\bigcup_{j=1}^N I_j .
\end{equation}
We define the DG approximation space of piecewise polynomials up to degree $k$ as
\begin{equation}
    \bbV^k_h = \{\psi_h: \psi_h\rvert_{I_j}\in \bbP_k(I_j), \quad \forall I_j\in\cT\},
\end{equation}
where $\bbP_{k}(I_j)$ denotes the space of polynomials of degree less than or equal to $k$ on the element $I_{j}$.

The weak formulation of Equation~\eqref{eq:linearadvectionPDE} is obtained by multiplying the PDE by a test function $\psi_h\in\bbV^{k}_{h}$, integrating the PDE over an element $I_j$, and then performing integration by parts on the PDE's spatial derivative terms.
The semi-discrete DG method is defined as follows: find $u_h\in\bbV^{k}_{h}$ such that for all $\psi_{h}\in\bbV^{k}_{h}$,
\begin{equation}
\label{eq:1DSemidiscreteDG}
    \innerproduct{\p_{t}u_h(t,x),\psi_h(x)}{I_{j}} + \Big[\hat{f}\big(u_h(t,x)\big)\psi_h(x)\Big]^{x_{j+\half}}_{x_{j-\half}} - \innerproduct{f\big(u_h(t,x)\big),\p_{x}\psi_h(x)}{I_{j}} = 0, 
\end{equation}
holds for each $I_j\in\cT$.
In Equation~\eqref{eq:1DSemidiscreteDG} we have defined the inner product and the surface term from integration by parts, respectively, as
\begin{equation}
    \innerproduct{f,g}{I_{j}} = \int_{I_{j}}f g\,dx
    \quad
    \text{and}
    \quad
    \big[f(x)\big]^{b}_{a} = f(b) - f(a).
\end{equation}
In the DG method the solution $u_h$ is allowed to be discontinuous across element boundaries.
Therefore we need a numerical flux $\hat{f}$ to approximate $f$ at element boundaries
\[
\hat{f}_{j+\frac12}=\hat{f}\big(u_h(t,x_{j+\half})\big) := \widehat{F}\big(u_h(t,x^{-}_{j+\half}),u_h(t,x^{+}_{j+\half})\big),
\]
where $x^{\pm}_{j+1/2}$ denote the left and right limits at the element interface $x_{j+\frac12}$, and $\widehat{F}(\cdot,\cdot)$ is the standard numerical flux, such as an upwind, Godunov, or Lax--Friedrichs flux.

We use a nodal DG representation of the solution $u_h$.
On the element $I_{j}$, we let $\{x^{(j)}_{i}\}_{i=1}^{k+1}$ be the set of Gauss nodes associated with a $(k+1)$-point Gauss quadrature rule, and let $\{w_i\}_{i=1}^{k+1}$ 
be the corresponding quadrature weights.
We then define the basis functions used to represent the DG solution as the Lagrange basis functions constructed on these $(k+1)$ Gauss nodes
\begin{equation}
\ell^{(j)}_{i} = \prod^{k+1}_{q=1,q\neq i} \frac{x-x^{(j)}_q}{x^{(j)}_{i}-x^{(j)}_{q}},
\end{equation}
where $\ell^{(j)}_{i}$ denotes the $i$-th basis function on the element $I_j$.
Using this notation we can express $u_h$ restricted to the element $I_j$ as
\begin{equation}
\label{eq:CellDGNodalsolution}
    u_h(t,x)\Big\rvert_{I_j}=\sum_{i=1}^{k+1}u^{(j)}_{i}(t)\ell^{(j)}_{i}(x),
    \quad
    x\in I_j.
\end{equation}
For concreteness, we employ Lagrange polynomials constructed from Gauss nodes; however, the mesh-based filtering approach is not restricted to this choice of polynomial basis functions.  

We evolve Equation~\eqref{eq:1DSemidiscreteDG} in time using an explicit Runge--Kutta (RK) method.
Let an explicit \(s\)-stage RK method be given by coefficients
\(\{a_{rq}\}_{1\le q<r\le s}\), weights \(\{b_r\}_{r=1}^s\), and nodes
\(\{c_r\}_{r=1}^s\).
Before applying the RK method to Equation~\eqref{eq:1DSemidiscreteDG} we rearrange the equation into the form
\[
    \innerproduct{\p_t u_h,\psi_h}{I_j}
    =
    \mathcal F_j(u_h;\psi_h),
\]
where
\begin{equation}
    \mathcal F_j(v_h;\psi_h)
    =
    \innerproduct{f(v_h),\p_x\psi_h}{I_j}
    - \hat f_{j+\half}(v_h)\psi_h(x_{j+\half}^-)
    + \hat f_{j-\half}(v_h)\psi_h(x_{j-\half}^+).
\end{equation}
Given \(u_h^n\), the RK stage approximations \(u_h^{[r]}\in\bbV_h^k\) are defined by
\begin{subequations}
\label{eq:RKTimeStepping}
\begin{equation}
    \innerproduct{u_h^{[r]},\psi_h}{I_j}
    =
    \innerproduct{u_h^n,\psi_h}{I_j}
    +
    \Delta t_n
    \sum_{q=1}^{r-1}
    a_{rq}\,
    \mathcal F_j(u_h^{[q]};\psi_h),
    \qquad r=1,\ldots,s,
\end{equation}
\begin{equation}
    \innerproduct{u_h^{n+1},\psi_h}{I_j}
    =
    \innerproduct{u_h^n,\psi_h}{I_j}
    +
    \Delta t_n
    \sum_{r=1}^{s}
    b_r\,
    \mathcal F_j(u_h^{[r]};\psi_h).
\end{equation}
\end{subequations}
Here $u^{n+1}_{h} = u_h(t_{n+1},\cdot)$ with $t_{n+1} = t_{n}+\Delta t_n$, and $u^{(r)}_{h} = u(t_n+c_r\Delta t_n,\cdot)$ is the $r$-th RK stage approximation.
The timestep $\Delta t_n$ must obey the following CFL condition for the linear case $f(u) = c\,u$ to ensure the stability of the solution
\begin{equation}
    |c|\frac{\Delta t_n}{\Delta x} \leq \mathrm{CFL}.
\end{equation}
This CFL condition must also be respected for nonlinear fluxes $f(u)$.  
Finally, if the solution becomes discontinuous or enforcing bounds on the solution is necessary, slope and bound-enforcing limiters (in that order) can be applied to the solution after every stage approximation of the RK method
\begin{subequations}
\label{eq:RKTimeSteppingwithLimiters}
\begin{equation}
    \innerproduct{u_h^{[r]},\psi_h}{I_j}
    =
    \Theta\left(\Phi\left(
    \innerproduct{u_h^n,\psi_h}{I_j}
    +
    \Delta t_n
    \sum_{q=1}^{r-1}
    a_{rq}\,
    \mathcal F_j(u_h^{[q]};\psi_h)
    \right)\right),
    \qquad r=1,\ldots,s,
\end{equation}
\begin{equation}
    \innerproduct{u_h^{n+1},\psi_h}{I_j}
    =
    \Theta\left(\Phi\left(
    \innerproduct{u_h^n,\psi_h}{I_j}
    +
    \Delta t_n
    \sum_{r=1}^{s}
    b_r\,
    \mathcal F_j(u_h^{[r]};\psi_h)
    \right)\right).
\end{equation}
\end{subequations}
Here $\Phi$ represents a slope limiter, and $\Theta$ represents a bound-enforcing limiter.

\subsection{DG Method with Mesh-Based Filtering}

\begin{figure}
     \centering
     \begin{subfigure}[b]{0.45\textwidth}
         \centering
         \includegraphics[width=\textwidth]{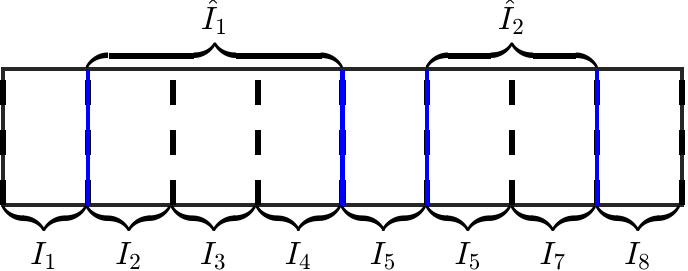}
     \end{subfigure}
     \begin{subfigure}[b]{0.45\textwidth}
         \centering
         \includegraphics[scale=0.75]{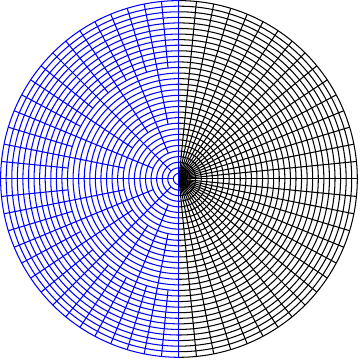}
     \end{subfigure}
        \caption{The left panel depicts a one-dimensional fine mesh of eight cells and its merging region consisting of two merged cells.
        Blue vertical lines indicate the boundaries of the two merged cells.
        Dashed black lines indicate the boundaries of fine cells.  
        The fine and merged cells have been labeled on the bottom and top axes, respectively.
        The right panel depicts a two-dimensional spherical-polar fine mesh (black) and its (reflected) merged mesh (blue).}
        \label{fig:ExampleGrids}
\end{figure}

We now introduce the mesh-based filtering procedure. 
\begin{dfn}
    Let $\cT$ denote the original structured triangulation of the domain $\Omega$.
    We refer to $\cT$ as the \textbf{fine mesh}.
    Let $\widehat{\cT}$ denote a triangulation of $\Omega$ such that each
element of $\widehat{\cT}$ is a union of one or more elements of $\cT$.
    We refer to $\widehat{\cT}$ as the \textbf{merged mesh}.
\end{dfn}
\noindent In general we will use the same notation as outlined in Section \ref{subsec:StandardDG}, and objects associated with the merged mesh $\widehat{\cT}$ will be denoted by a hat.  

\begin{dfn}
    Let \(\Omega_2\subset\Omega\) denote the union of all elements of
$\widehat{\cT}$ that contain two or more elements of $\cT$.
    We refer to $\Omega_2$ as the \textbf{merging region}.
    The elements in $\Omega_2$ can be disjoint.
    Let \(\Omega_1=\Omega\setminus\Omega_2\) denote the union of all elements of
$\widehat{\cT}$ that coincide with single elements of $\cT$.
    We refer to $\Omega_1$ as the \textbf{non-merging region}.
\end{dfn}
\noindent We assume that \(\Omega_1\) is a union of fine cells and that \(\Omega_2\) is partitioned into $\widehat{N}$ disjoint merged cells
\begin{equation}
\label{eq:Omega2_MergedCell_Def}
    \Omega_2
    =
    \bigcup_{p=1}^{\widehat N}\widehat I_p,
    \qquad
    \widehat I_p
    =
    \bigcup_{m\in A_{p}} I_{m}.
\end{equation}
For each \(p\) there is an indexed set $A_p\subset\{1,\ldots,N\}$ which collects the fine cells \(I_{m}\) comprising \(\widehat I_p\).
These fine cells must be contiguous.
We define \(|A_p| \vcentcolon= M_p\geq 1\) as the number of fine cells in the merged cell \(\widehat I_p\), and it may vary with \(p\).
The merged mesh $\widehat{\cT}$ is defined as
\begin{equation}
    \widehat{\cT}
    =
    \{ I_j: I_j\subset\Omega_1\}
    \cup
    \{ \widehat I_p: p=1,\ldots,\widehat N\},
    \quad \overline{\Omega} = \Omega_1\cup\Omega_2.
\end{equation}

To illustrate the relationship between the fine mesh $\cT$ and the merged mesh $\widehat{\cT}$, we present one- and two-dimensional examples in Figure~\ref{fig:ExampleGrids}.  
The left panel depicts a one-dimensional domain consisting of eight fine cells grouped into merged cells of varying sizes with a nonempty $\Omega_1$.
In this example, $\Omega_1 = I_1\cup I_5\cup I_8$, and $\Omega_2 = \widehat{I}_1\cup\widehat{I}_2$, where $\widehat{I}_1 = I_2\cup I_3 \cup I_4$ and $\widehat{I}_2 = I_6\cup I_7$.  
Consequently, $M_1 = 3$ and $M_2 = 2$.  
In practice, merged cells are constructed according to prescribed merging criteria and may therefore contain different numbers of fine cells.  
The right panel shows a two-dimensional spherical-polar example.  
The black mesh denotes $\cT$, while the blue mesh illustrates $\widehat{\cT}$ after reflection about the polar axis.  
Here, cells are merged to obtain roughly uniform element aspect ratios.  
We emphasize that $\widehat{\cT}$ is introduced only to define the mesh-based filter: all RKDG evolution is performed on the fine mesh $\cT$, while the stable time step is determined by the merged mesh $\widehat{\cT}$.

The solution approximation space on the merged mesh $\widehat{\cT}$ is
\begin{equation}
    \widehat{\bbV}^k_h 
    =\{v_h: v_h\big|_K\in\bbP_k(K), \,
    \forall K\in\widehat{\cT} \}
    = \{\psi_h: \psi_h\rvert_{I_j}\in \bbP_k(I_j), \forall I_j\in\Omega_1\} \cup \{\hat{\psi}_h: \hat{\psi}_h\rvert_{\widehat{I}_p}\in \bbP_k(\widehat{I}_p), \forall \widehat{I}_p\in\Omega_2\}.
\end{equation}
The corresponding DG solution $\hat{u}_h\in\widehat{\bbV}^{k}_{h}$ is given by
\begin{equation}
\label{eq:cellDGNodalsolution_merge}
    \hat{u}_{h}(t,x)
    = 
    \begin{cases}
    &\displaystyle\sum_{i=1}^{k+1}u^{(j)}_{i}(t)\ell^{(j)}_{i}(x) \quad \text{if} \quad x\in I_{j}\subset\Omega_1,\\
    &\displaystyle\sum_{i=1}^{k+1}\hat{u}^{(p)}_{i}(t)\hat{\ell}^{(p)}_{i}(x) \quad \text{if} \quad x\in\widehat{I}_{p}\subset\Omega_2.
    \end{cases}
\end{equation}
Here $\hat{\ell}^{(p)}_{i}$ denotes the $i$-th Lagrange basis function on the merged cell \(\widehat I_p\), formed from the $(k+1)$ Gauss nodes $\{\hat{x}^{(p)}_{i}\}^{k+1}_{i=1}\subset\widehat{I}_{p}$.
The left panel in the top row of Figure \ref{fig:linearbases} depicts linear $\ell^{(j)}_{i}$ and $\hat{\ell}^{(p)}_{i}$ for an element $\widehat{I}_{p}$ consisting of two fine cells.
On the unmerged cells in $\Omega_1$, \(\widehat u_h\) has the same local representation as the fine-mesh solution \(u_h\).

\begin{figure}
     \centering
     \begin{subfigure}[b]{0.45\textwidth}
         \centering
         \includegraphics[width=\textwidth]{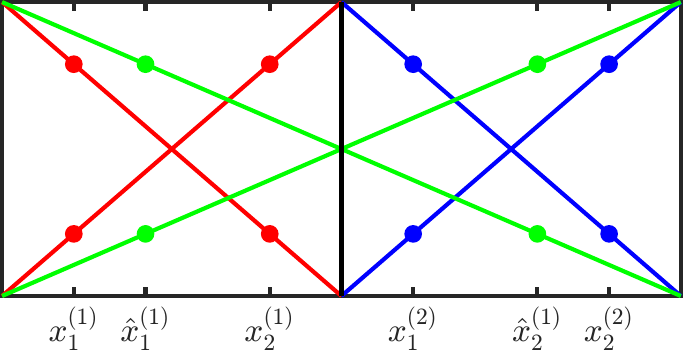}
         \label{subfig:linearbases}
     \end{subfigure}
     \hfill
     \begin{subfigure}[b]{0.45\textwidth}
         \centering
         \includegraphics[width=\textwidth]{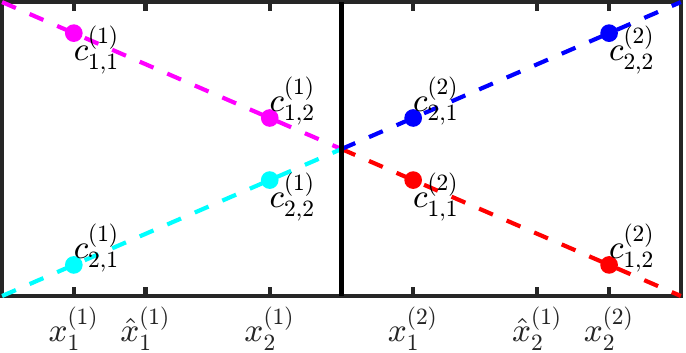}
         \label{subfig:locallinearmergebasis}
     \end{subfigure}
     \begin{subfigure}[b]{0.45\textwidth}
         \centering
         \includegraphics[width=\textwidth]{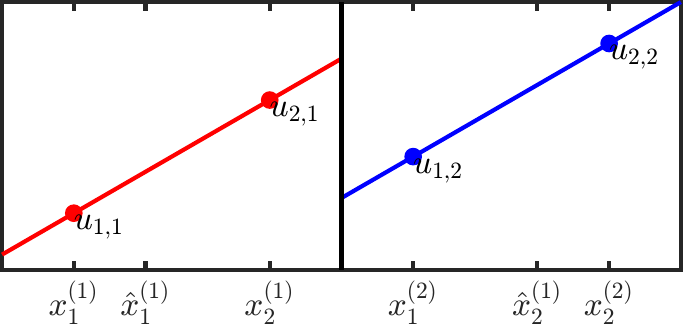}
         \label{subfig:examplesolution}
     \end{subfigure}
     \hfill
     \begin{subfigure}[b]{0.45\textwidth}
         \centering
         \includegraphics[width=\textwidth]{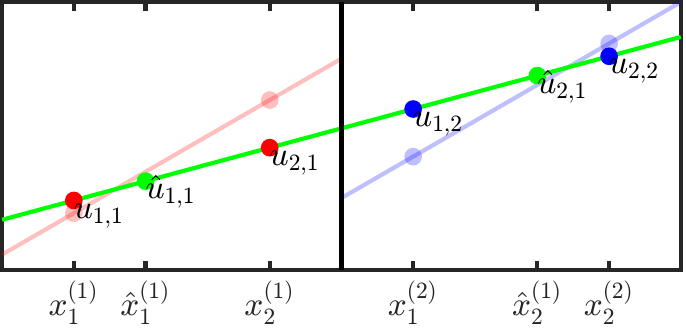}
         \label{subfig:examplesolutionmerged}
     \end{subfigure}
        \caption{Top row: example of local linear bases (left) on two cells (red and blue) and the basis for the two cells when merged together (green), and the merged basis functions (right) expressed locally on each cell.
        The nodal values are the cell merging weights as defined in Section \ref{subsec:MergingWeights}.
        Bottom row: example of a linear DG solution across two cells (left), and the same solution (right) after the projection--prolongation step.
        The green nodes are the nodal values on the merged cell, while the red and blue nodes are the nodal values restricted to the fine cells.}
        \label{fig:linearbases}
\end{figure}

Let \(\Delta t\) denote a stable time step for the fine mesh \(\cT\), and let \(\widehat{\Delta t}\) denote the corresponding stable time step associated with the merged mesh \(\widehat{\cT}\).
In the ideal case, $\widehat{\Delta t}\gg\Delta t$.
Advancing the fine-mesh DG solution directly with an ideal \(\widehat{\Delta t}\) would generally violate the fine-mesh CFL restriction.
The purpose of the mesh-based filtering strategy is to evolve the fine-mesh solution $u_h$ using the larger time step size $\widehat{\Delta t}$ by projecting it onto the merged mesh and prolonging onto the fine mesh after each timestep. 
In one dimension, we will show that this filtered fine-mesh update is equivalent to an update performed on the merged mesh.

The key steps of the mesh-based filtering process are as follows
\begin{steps}
    \item Determine the merged mesh $\widehat{\cT}$ based on the original mesh $\cT$.
    \item Project the fine-mesh solution at the current time level onto $\widehat{\cT}$ and then prolong the resulting merged-mesh polynomial back to the fine mesh $\cT$.
    \item Evolve the degrees of freedom on $\cT$ with the timestep determined by $\widehat{\cT}$. 
    \item Repeat Steps 2 and 3 until the solution is evolved to the final time. 
    For multi-stage Runge--Kutta methods, the projection--prolongation operation (Step 2) must be applied at each inner stage.
\end{steps}
Figure \ref{fig:MergingFlowchart} provides a flowchart of the mesh-based filter algorithm with Forward Euler time integration incorporating the above steps.
\begin{figure}
    \centering
    \includegraphics[scale=0.7]{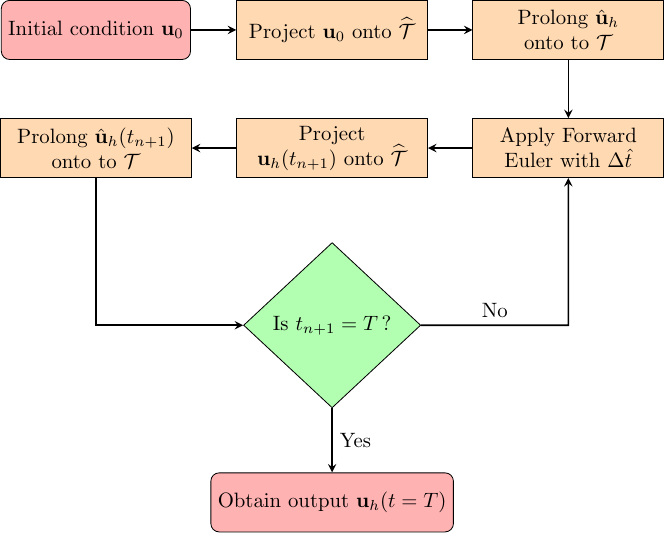}
    \caption{A flowchart describing the mesh-based filtering process using Forward Euler.}
    \label{fig:MergingFlowchart}
\end{figure}

To define the filter precisely and make the algorithm explicit, we define the following projection operators $\cP$ and $\widehat{\cP}$:
\begin{subequations}
\begin{equation}
\label{eq:FineCellProjection}
    \int_{I_{j}}\cP[u]\,\ell^{(j)}_{i}\,dx
    =
    \int_{I_{j}}u\,\ell^{(j)}_{i}\,dx
    \quad
    \text{holds for all}
    \quad
    i=1,\ldots,k+1
    \quad
    \text{and all}
    \quad
    I_j\in\Omega\in\cT,
\end{equation}
\begin{equation}
\label{eq:MergedCellProjection}
    \int_{\widehat{I}_{p}}\widehat{\cP}[u]\,\hat{\ell}^{(p)}_{i}\,dx
    =
    \int_{\widehat{I}_{p}}u\,\hat{\ell}^{(p)}_{i}\,dx
    \quad
    \text{holds for all}
    \quad
    i=1,\ldots,k+1
    \quad
    \text{and all}
    \quad
    \widehat{I}_p\in\Omega_2.
\end{equation}
\end{subequations}
The operator $\cP$ defines the $L^2$ projection onto an element $I_j\subset\cT$ in the fine space, while the operator $\widehat{\cP}$ defines the $L^2$ projection onto an element $\widehat{I}_{p}\subset\Omega_2$ in the merged space.
Then for any fine cell $I_{m}\subset\widehat{I}_p$, the projection--prolongation update (Step 2) of the algorithm is done by updating $u_h$ with the following equation
\begin{equation}
    \label{eq:ProjectionProlongationUpdate}
    u_h\Big\rvert_{I_{m}} = \widehat{\cP}[u_h]\Big\rvert_{I_{m}}.
\end{equation}
That is, \(u_h\) is first projected onto the polynomial space on the merged cell \(\widehat I_p\), and the resulting polynomial is then restricted back to each fine cell contained in \(\widehat I_p\).
This operation makes $u_h$ continuous across each $\widehat{I}_p$.

In the next subsection we show, under the assumptions stated there, that advancing the fine-mesh DG solution by one Forward Euler step with time step \(\widehat{\Delta t}\), followed by the projection--prolongation update \eqref{eq:ProjectionProlongationUpdate}, is equivalent to advancing the DG solution directly on the merged mesh \(\widehat{\cT}\) with the same time step.

\subsection{Projection--Prolongation Filtering Formula and Equivalence}

We now derive the nodal form of the projection--prolongation operator and prove that employing it with the RKDG method on the fine mesh is equivalent, in one dimension, to evolving on the merged mesh.  
Throughout this subsection we assume that all projection and DG volume integrals are evaluated exactly, or with quadrature rules that are consistent with the nodal mass matrices used below. 
We also assume that no limiter is applied unless explicitly stated.

\subsubsection{Projection onto a Merged Cell}

On a merged cell \(\hat I_p\), the projection of \(u_h\) onto \(\bbP_k(\hat I_p)\) is written as
\begin{equation}
\label{eq:DGMergeRepresentation}
    \widehat{\cP}u_h\big|_{\hat I_p}
    =
    \hat u_h\big|_{\hat I_p}
    =
    \sum_{i=1}^{k+1}\hat u_i^{(p)}\hat\ell_i^{(p)} .    
\end{equation}
Consider the nodal DG representation $u_h$ with Equation~\eqref{eq:CellDGNodalsolution}.
Using the definition of \(\widehat{\cP}\) in Equation~\eqref{eq:MergedCellProjection}, the nodal coefficients are
\begin{equation}
    \label{eq:MergedNodalValue}
    \hat{u}^{(p)}_{i}
    =
    \frac{\int_{\widehat{I}_p}u_h\,\hat{\ell}^{(p)}_{i}\,dx}{\widehat{V}^{(p)}_{i}},
    \quad \text{where} \quad
    \widehat{V}^{(p)}_{i} = \int_{\widehat{I}_{p}}\hat{\ell}^{(p)}_{i}\hat{\ell}^{(p)}_{i}\,dx.
\end{equation}
This expression for the nodal coefficients is arrived at by taking advantage of the orthogonality of Lagrange polynomials constructed with Gauss nodes.
Otherwise solving for the coefficients will require inverting the mass matrix.
The quantity $V^{(j)}_{i}$ is defined the same way as $\widehat{V}^{(p)}_{i}$ but using $\ell^{(j)}_{i}$ and integrating over the cell $I_{j}$
\begin{equation}
    V^{(j)}_{i} = \int_{I_j} \ell^{(j)}_{i}\ell^{(j)}_{i}\,dx.
\end{equation}

\subsubsection{The Cell Merging Weights}

\label{subsec:MergingWeights}

Equation~\eqref{eq:MergedNodalValue} provides a formula for the nodal values of the solution projected onto the merged cell.  
To prolong $\widehat{\cP}u_h\big|_{\hat I_p}$ in Equation~\eqref{eq:DGMergeRepresentation} back to the fine mesh, we express $\hat{\ell}^{(p)}_{i}$ as a linear combination of the local fine basis functions.
We do this with what we call cell merging weights.
Considering the fine cell $I_{m}\subset \widehat{I}_p$, we desire
\begin{equation}
\label{eq:localMergedBasis}
    \hat{\ell}^{(p)}_{i}
    =
    \sum_{\alpha=1}^{k+1}c^{(m)}_{i,\alpha}\ell^{(m)}_{\alpha}
    \quad
    \text{for all}
    \quad
    i=1,\ldots,k+1,
\end{equation}
where $c^{(m)}_{i,\alpha}$ is the weight associated to cell $I_{m}$, fine basis function $\ell^{(m)}_{\alpha}$, and merged basis function $\hat{\ell}^{(p)}_{i}$.
Since we use Lagrange polynomials for $\ell^{(m)}_{\alpha}$, these weights are easily determined by evaluating $\hat{\ell}^{(p)}_{i}$ at a fine mesh node (and using the cardinal property)
\begin{equation}
\label{eq:MergingWeights1D}
    c^{(m)}_{i,\alpha}
    =
    \hat{\ell}^{(p)}_{i}(x^{(m)}_{\alpha}),
\end{equation}
where $x^{(m)}_{\alpha}$ is the $\alpha$-th fine-mesh node on $I_{m}$.
The right panel of the top row of Figure \ref{fig:linearbases} depicts this relation between the cell merging weights and the merged basis functions.
Note that $c^{(m)}_{i,\alpha}$ inherently depends on the index $p$ of the merged cell through Equation~\eqref{eq:MergingWeights1D}, meaning the weights can differ from merged cell to merged cell.
It should also be noted that if the fine mesh is uniform, then the cell merging weights need only be computed for each distinct value of $M_p$ (recall that $M_p$ represents the number of fine cells a merged cell spans).
Taking advantage of this property may greatly reduce the amount of memory used.

\subsubsection{Prolonging onto a Fine Cell}

The final step is to prolong $\widehat{\cP}u_h\big|_{\hat I_p}$ in Equation~\eqref{eq:DGMergeRepresentation} back to the fine mesh $\cT$ by using Equation~\eqref{eq:localMergedBasis}.
We first split the integral in Equation~\eqref{eq:MergedNodalValue} over the $M_p$ contiguous fine cells which comprise $\widehat{I}_p$ and use Equation~\eqref{eq:localMergedBasis} to express $\hat{\ell}^{(p)}_{i}$ locally
\begin{equation}
    \hat{u}^{(p)}_{i}
    =
    \frac{1}{\widehat{V}^{(p)}_{i}}
    \sum_{m=1}^{M_p}
    \int_{I_{m}}
    \bigg( \sum_{\alpha=1}^{k+1}u^{(m)}_{\alpha}\ell^{(m)}_{\alpha}\bigg)
    \bigg( \sum_{\beta=1}^{k+1}c^{(m)}_{i,\beta}\ell^{(m)}_{\beta}\bigg)
    \,dx.
\end{equation}
Here $u_h$ has also been replaced by its nodal representation on the element $I_{m}$.
We then evaluate each integral using the $(k+1)$ Gauss nodes and weights.
Due to the discrete orthogonality of the local basis, the majority of the terms disappear and the result is
\begin{equation}
\label{eq:localMergedNodalValue}
   \hat{u}^{(p)}_{i}
    =
    \sum_{m=1}^{M_p} \sum_{\alpha=1}^{k+1}\frac{c^{(m)}_{i,\alpha}V^{(m)}_{\alpha}}{\widehat{V}^{(p)}_{i}}u^{(m)}_{\alpha}.
\end{equation}
Finally we use Equations~\eqref{eq:localMergedBasis} and \eqref{eq:localMergedNodalValue} in Equation~\eqref{eq:DGMergeRepresentation} to update the nodal DG solution on the fine cell $I_{n}\subset \widehat{I}_p$
\begin{equation}
\label{eq:UpdatedCellDGNodalsolution}
    \hat{u}_h\Big\rvert_{I_{n}}
    =
    \sum_{i=1}^{k+1} \hat{u}^{(p)}_{i} \hat{\ell}^{(p)}_{i} \Big\rvert_{I_{n}}
    =
    \sum_{\beta=1}^{k+1}
    \bigg(
    \sum_{m=1}^{M_p}\sum_{i=1}^{k+1}\sum_{\alpha=1}^{k+1}
    \frac{c^{(m)}_{i,\alpha}c^{(n)}_{i,\beta}V^{(m)}_{\alpha}}{\widehat{V}^{(p)}_{i}}
    u^{(m)}_{\alpha}
    \bigg)
    \ell^{(n)}_{\beta}.
\end{equation}
Comparing Equation~\eqref{eq:UpdatedCellDGNodalsolution} to Equation~\eqref{eq:CellDGNodalsolution}, we see that updated expansion coefficients on the fine cell $I_{n}$ can be expressed in terms of the original expansion coefficients as
\begin{equation}
\label{eq:UpdatedNodalValues}
    \tilde u^{(n)}_{\beta}
    =
    \sum_{m=1}^{M_p}\sum_{i=1}^{k+1}\sum_{\alpha=1}^{k+1}
    \frac{c^{(m)}_{i,\alpha}c^{(n)}_{i,\beta}V^{(m)}_{\alpha}}{\widehat{V}^{(p)}_{i}}
    u^{(m)}_{\alpha}.
\end{equation}
The projection--prolongation step (Step 2 of the algorithm) is implemented by replacing \(u_\beta^{(j_n)}\) with \(\tilde u_\beta^{(j_n)}\) for every fine cell \(I_{j_n}\subset\hat I_p\).
This process is thus repeated for all merged cells $\widehat{I}_{p}\in\Omega_2$.
\begin{rem}
	If \(M_p=1\) and $p\in A_{p}=\{j\}$, then \(\hat I_p=I_j\), \(\hat\ell_i^{(p)}=\ell_i^{(j)}\), and \(c_{i,\alpha}^{(j)}=\delta_{i\alpha}\). 
	Hence Equation~\eqref{eq:UpdatedNodalValues} reduces to \(\tilde u_\beta^{(j)}=u_\beta^{(j)}\).  
	Thus, the filter is the identity operator on unmerged cells.
\end{rem}

The bottom row of Figure \ref{fig:linearbases} shows a simple example of the merging process on a linear DG solution using two fine cells and one merged cell.
The left panel shows the solution after it has been updated from $u_{h}(t,x)$ to $u_{h}(t+\widehat{\Delta t},x)$ for some timestep $\widehat{\Delta t}$ determined by the merged mesh.
The right panel shows the solution after the projection--prolongation step.
The projected solution is plotted in green.
Note that the projection will make the solution continuous across the fine cells within the merged cell.
The solution may still be discontinuous across adjacent merged cells.  
The updated nodal values are indicated by the red and blue dots on the green line.

\subsubsection{Analysis on the Mesh-Based Filter}

To prove our claim that, under certain assumptions, advancing the fine-mesh DG solution by one Forward Euler step with time step \(\widehat{\Delta t}\), followed by the projection--prolongation update \eqref{eq:ProjectionProlongationUpdate}, is equivalent to advancing the DG solution directly on the merged mesh \(\widehat{\cT}\) with the same time step, we first prove the following lemma regarding how to set the initial condition.

\begin{lem}
\label{lem:InitialConditionEquivalence}
    For any \(u\in L^2(\Omega)\),
\[
    \widehat{\cP}[\cP [u]] = \widehat{\cP}[u] .
\]
    In particular, if the initial condition is projected onto $\cT$ and then projected onto $\widehat{\cT}$, this is equivalent to projecting directly onto $\widehat{\cT}$.
\end{lem}
\begin{proof}
It is sufficient to prove the claim for a single merged cell \(\hat I_p\). 
Let \(\hat\psi\) be an arbitrary polynomial in the space $\bbP_k(\hat I_p)$. 
Since \(\hat\psi|_{I_{j_m}}\in\bbP_k(I_{j_m})\) for each fine cell \(I_{j_m}\subset\hat I_p\), the definition of the fine-cell projection \(\cP\) gives
\[
    \int_{I_{j_m}}(\cP u-u)\hat\psi\,dx=0,
    \qquad m=1,\ldots,M_p .
\]
Summing over the fine cells in \(\hat I_p\) yields
\[
    \int_{\hat I_p}(\cP u-u)\hat\psi\,dx=0 .
\]
Therefore \(\cP u\) and \(u\) have the same moments against every basis function of \(\bbP_k(\hat I_p)\), which proves \(\widehat{\cP}[\cP [u]]=\widehat{\cP}[u]\).
\end{proof}

\begin{thm}
\label{thm:SolutionEquivalence}
Assume that, at time \(t_n\), the fine-mesh solution \(u_h^n\) is the prolongation of the merged-mesh solution \(\hat u_h^n\), that is
\[
    u_h^n\big|_{I_{j_m}}
    =
    \hat u_h^n\big|_{I_{j_m}}
    \qquad
    \text{for every } I_{j_m}\subset\hat I_p .
\]
Let \(u_h^{n+1,*}\) be obtained by applying one Forward Euler step on \(\cT\) with the merged-mesh time step \(\widehat{\Delta t}\), and let \(\hat u_h^{n+1}\) be obtained by applying one Forward Euler step on \(\widehat{\cT}\) with the same time step. 
Then
\begin{equation}
    \widehat{\cP}u_h^{n+1,*}
    =
    \hat u_h^{n+1}.
    \label{eq:theorem_equivalence}
\end{equation}
Therefore, the fine-mesh update after the projection--prolongation step is equivalent to the merged-mesh DG update.    
\end{thm}
\begin{proof}
	For any merged cell \(\hat I_p\), consider a merged basis function \(\hat\ell_i^{(p)}\). 
	Since \(\hat\ell_i^{(p)}|_{I_{j_m}}\in\bbP_k(I_{j_m})\), we can use it as a test function in the fine-mesh DG update on each fine cell \(I_{j_m}\subset\hat I_p\). 
	Summing the Forward Euler update over these fine cells gives
	\[
    		\innerproduct{u_h^{n+1,*},\hat\ell_i^{(p)}}{\hat I_p}
    		=
    		\innerproduct{u_h^n,\hat\ell_i^{(p)}}{\hat I_p}  
    		+\widehat{\Delta t}
    		\sum_{m=1}^{M_p}
    		\innerproduct{f(u_h^n),\p_x\hat\ell_i^{(p)}}{I_{j_m}}
    		-\widehat{\Delta t}\,\mathcal B_p ,
	\]
	where \(\mathcal B_p\) is the sum of the numerical-flux boundary terms over the fine cells in \(\hat I_p\). 
	The contributions on internal fine cell interfaces cancel because the numerical flux is single-valued and \(\hat\ell_i^{(p)}\) is continuous across those interfaces. 
	Hence only the two boundary contributions at \(\partial\hat I_p\) remain.  
	Using the assumption \(u_h^n=\hat u_h^n\) on \(\hat I_p\), we obtain
	\[
    		\innerproduct{u_h^{n+1,*},\hat\ell_i^{(p)}}{\hat I_p}
    		=
    		\innerproduct{\hat u_h^n,\hat\ell_i^{(p)}}{\hat I_p}
    		+\widehat{\Delta t}
    		\innerproduct{f(\hat u_h^n),\p_x\hat\ell_i^{(p)}}{\hat I_p} 
    		-\widehat{\Delta t}
    		\Big[
    			\hat f(\hat u_h^n)\hat\ell_i^{(p)}
    		\Big]_{x_{p-\half}}^{x_{p+\half}} .
	\]
	This is precisely the weak form of the Forward Euler update on the merged cell \(\hat I_p\). 
	Therefore
	\[
    		\innerproduct{\widehat{\cP}u_h^{n+1,*},\hat\ell_i^{(p)}}{\hat I_p}
    		=
    		\innerproduct{u_h^{n+1,*},\hat\ell_i^{(p)}}{\hat I_p}
    		=
    		\innerproduct{\hat u_h^{n+1},\hat\ell_i^{(p)}}{\hat I_p}
	\]
	for all \(i=1,\ldots,k+1\). 
	By uniqueness of the projection in \(\bbP_k(\hat I_p)\), the two polynomials, $\widehat{\cP}u_h^{n+1,*}$ and $\hat u_h^{n+1}$, are equal on the merged cell \(\hat I_p\).
	Repeating the argument for every merged cell, we conclude Equation~\eqref{eq:theorem_equivalence}.
\end{proof}
\begin{rem}
	Lemma~\ref{lem:InitialConditionEquivalence} provides the equality assumption in Theorem~\ref{thm:SolutionEquivalence} at \(t=0\). 
	The theorem then propagates this equality to any $t_n$ by induction after each projection--prolongation step.
\end{rem}
\begin{rem}[Explicit Runge--Kutta equivalence]
	Theorem~\ref{thm:SolutionEquivalence} is stated for a Forward Euler step for clarity. 
	The same argument extends to any explicit RK method, provided that the projection--prolongation operation is applied to each inner stage solution before the DG residual is evaluated. 
	If the fine-mesh stage solution is the prolongation of the corresponding merged-mesh stage solution, then the argument in the proof of Theorem~\ref{thm:SolutionEquivalence} shows that the projected fine-mesh residual equals the merged-mesh residual. 
	Applying this argument stage by stage gives
	\(\widehat{\cP}u_h^{[r]} = \widehat u_h^{[r]}, \, r=1,\ldots,s, \)
	and hence \( \widehat{\cP}u_h^{n+1} = \widehat u_h^{n+1}\).
	Thus the filtered fine-mesh RK update is equivalent to the RK update on the merged mesh under the same assumptions as in Theorem~\ref{thm:SolutionEquivalence}.
\end{rem}
\begin{rem}
	If a source term \(s(u)\) is included in the model, the same argument holds, provided it is evaluated with the same quadrature and representation on the fine and merged meshes. 
\end{rem}
\begin{rem}\label{rem:Limiters}
    If a bound-enforcing or slope limiter is applied on the merged representation after projection onto $\widehat{\cT}$ and before prolongation to $\cT$, then equivalence between $u_{h}$ and $\hat{u}_{h}$ will still hold. 
    However if the limiter is applied directly on the fine mesh $\cT$ when mesh-based filtering is used, then the equivalence in Theorem~\ref{thm:SolutionEquivalence} may be lost. 
\end{rem}

\subsection{Numerical Examples in One Dimension}

To give a simple demonstration of the mesh-based filtering process and validate the
projection--prolongation procedure, we present numerical results for the one-dimensional linear transport equation and nonlinear Burgers' equation, which correspond to $f(u(t,x)) = u(t,x)$ and $f(u(t,x))= \frac{1}{2}u^2(t,x)$, respectively, in Equation~\eqref{eq:linearadvectionPDE}.

We set the total number of fine mesh cells as $N = \{20,40,80,160,320,640\}$.
When we simulate with the mesh-based filtering process, we merge cells in groups of two so that
\[
\widehat{I}_p = I_{2p-1}\cup I_{2p}, \quad \text{for} \quad p = 1,\ldots,N/2.
\]
As a result, when we simulate with $N$ cells, we expect the mesh-based filtering solution to be comparable to a solution with $N/2$ cells which does not use mesh-based filtering.
This criterion for merging was chosen only for its simplicity.

\begin{table}[ht]
    \centering
    \begin{tabular}{crrrrrrr}
        \tableline
        & \multicolumn{3}{c}{First Order ($k=0$)}
        & \multicolumn{2}{c}{Second Order ($k=1$)}
        & \multicolumn{2}{c}{Third Order ($k=2$)}\\
        \tableline
        $N$ &Fine Mesh &Mesh Filter &Fine Mesh+$\widehat{\Delta t}$ &Fine Mesh &Mesh Filter &Fine Mesh &Mesh Filter\\
        \tableline
        20  &7.3422e-02  &1.4448e-01 &3.389e\texttt{+}01 &8.8216e-03 &3.7545e-02 &9.8042e-05 &8.1029e-04 \\
        \tableline
        40  &3.6393e-02 &7.3422e-02 &1.1403e\texttt{+}01 &2.1374e-03 &8.8216e-03 &1.2161e-05 &9.8042e-05  \\
        \tableline
        80  &1.7580e-02 &3.6393e-02 &9.3767e\texttt{+}02 &5.3140e-04 &2.1374e-03 &1.5173e-06 &1.2161e-05 \\
        \tableline
        160 &8.7902e-03 &1.7580e-02 &2.0302e\texttt{+}022 &1.3421e-04 &5.3140e-04 &1.8958e-07 &1.5173e-06 \\
        \tableline
        320 &4.3867e-03 &8.7902e-03 &1.5589e\texttt{+}060 &3.3697e-05 &1.3421e-04 &2.3698e-08 &1.8958e-07 \\
        \tableline
        640 &2.1830e-03 &4.3867e-03 &1.8446e\texttt{+}136 &8.5447e-06 &3.3697e-05 &2.9622e-09 &2.3698e-08 \\
        \tableline
    \end{tabular}
    \caption{$L^2$ error of the linear transport solution at time $t = 1$ when using first, second, and third order DG methods. The errors of the fine mesh and mesh filter solutions are compared.
    The error of the mesh-based filtering solution is computed by projecting the solution onto $\widehat{\cT}$.
    For the first order method, we include results when the fine-mesh solution is advanced with $\widehat{\Delta t}$ without applying the mesh-based filter (column labeled ``Fine Mesh+$\widehat{\Delta t}$'').
    }
    \label{tab:LinearTransport_k=0_p=1}
\end{table}

\begin{figure}
    \centering
    \centering
     \begin{subfigure}[b]{0.45\textwidth}
         \centering
         \includegraphics[width=\textwidth]{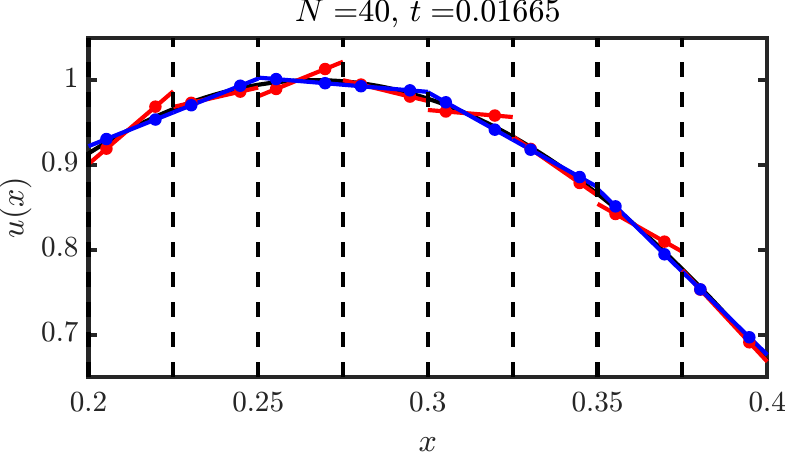}
         \label{subfig:mergingprocess}
     \end{subfigure}
     \hfill
     \begin{subfigure}[b]{0.45\textwidth}
         \centering
         \includegraphics[width=\textwidth]{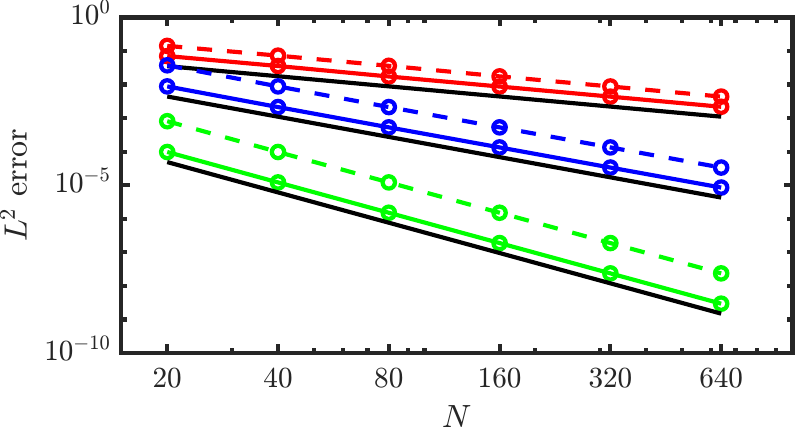}
         \label{subfig:LinearTransportOrder}
     \end{subfigure}
    \caption{
    The left panel shows the linear transport solution on a subset of $\Omega$ after one mesh-based filtering step, using explicit midpoint time stepping and linear polynomials. 
    The black line denotes the exact solution, the red lines the solutions advanced with time step $\widehat{\Delta t}$, and the blue lines the updates obtained from Equation~\eqref{eq:UpdatedNodalValues}. Markers indicate nodal values.
    The right panel shows the $L^2$ error of the linear transport solution at $t=1$ for the values $N$ tested. Solid and dashed lines correspond to solutions without and with mesh-based filtering, respectively. Red, blue, and green lines correspond to constant polynomials with Forward Euler, linear polynomials with the midpoint rule, and quadratic polynomials with SSPRK3, respectively. Black lines indicate first-, second-, and third-order reference slopes.
    }
    \label{fig:lineartransport_N40_N20}
\end{figure}

\subsubsection{Linear Transport Equation}
For the linear transport equation, we let $\Omega = [0,1]$, use periodic boundary conditions and set the initial condition as $u_0(x) = \sin(2\pi x)$.
The final time is $T = 1$, so that the exact solution equals the initial condition.
The time steps for a solution without merging and with merging, respectively, are
\[
	\Delta t = \mathrm{CFL}\,\min_{j}\{|I_j|\}
	\quad
	\text{and}
	\quad
	\widehat{\Delta t} = \mathrm{CFL}\,\min_{p}\{|\widehat{I}_p|\}.
\]
For the pairwise merging used here, $\widehat{\Delta t} = 2\Delta t$.
We set the value of $\mathrm{CFL}$ in accordance with Table 2.2 in \citet{cockburnShu_2001} for the corresponding polynomial degree and time integrator.

We perform the computation using DG methods with $k=0$, $1$, $2$.
The first-, second-, and third-order calculations use constant polynomials with Forward Euler time stepping, linear polynomials with explicit midpoint rule, and quadratic polynomials with SSPRK3 \citep{ShuOsher_1988} time stepping, respectively.
The right panel of Figure \ref{fig:lineartransport_N40_N20} shows that, for these first, second, and third order methods, we observe the expected order accuracy when using the mesh-based filtering process. 

The \(L^2\) errors of the DG solutions at \(t=1\) with $k=0$, $1$, $2$ are also reported in Table~\ref{tab:LinearTransport_k=0_p=1}.
The $L^2$ error of the fine-mesh solution is computed by evaluating $\|u_h - u\|_2$ using the Gauss quadrature rule defined by the DG nodes on $\cT$, while the error of the mesh-filtering solution is computed by evaluating $\|\widehat{\cP}[u_h]-u\|_2$ using the Gauss quadrature rule defined by the DG nodes on $\widehat{\cT}$.
In this way, the filtered computation with \(N\) fine cells has the same error as the fine-mesh computation with \(N/2\) cells.  
To ensure the errors are consistent with Theorem \ref{thm:SolutionEquivalence} we compute
    $\int_{I_{j}}\sin(2\pi x)\ell^{(j)}_{i}\,dx$
    and 
    $\int_{\widehat{I}_p}\sin(2\pi x)\hat{\ell}^{(p)}_{i}\,dx$
exactly when projecting the initial condition $u_0(x)$ onto the DG solution space. 
Table~\ref{tab:LinearTransport_k=0_p=1}, in the column labeled ``Fine Mesh+$\widehat{\Delta t}$'', also shows the result of evolving the fine-mesh DG solution with the larger time step \(\widehat{\Delta t}\) without applying the projection--prolongation filter.  
Because $\widehat{\Delta t}$ exceeds the fine-mesh CFL limit, the method is unstable and the solution blows up. 
In contrast, applying the filter after each update recovers the merged-mesh evolution described by Theorem~\ref{thm:SolutionEquivalence}, and yields the stable method with the larger time step \(\widehat{\Delta t}\).

The left panel of Figure \ref{fig:lineartransport_N40_N20} depicts the solution obtained using the mesh-based filtering process after one time step of a second order method using $N = 40$ cells.
The red lines represent the fine-mesh DG solution after one update with the larger time step \(\widehat{\Delta t}\), before applying the mesh-based filter as in Equation~\eqref{eq:UpdatedNodalValues}.
The blue lines represent the DG solution after applying Equation~\eqref{eq:UpdatedNodalValues}.
The mesh-based filter has a stabilizing effect on the solution, otherwise the solution would not stably evolve with $\widehat{\Delta t}$.
Additionally, since Equation~\eqref{eq:UpdatedNodalValues} projects the red line solution onto the merged mesh, we see that the blue line solution becomes continuous across cells within a merged cell.

\begin{table}
    \centering
    \begin{tabular}{crcrc}
         \tableline
         $N$ & Fine Mesh & Order & Mesh Filter & Order \\
         \tableline
         20  & 6.7083e-04 & --    & 2.1052e-03  & --    \\
         \tableline
         40  & 9.4263e-05 & 2.831 & 6.7083e-04 & 1.650 \\
         \tableline
         80  & 1.3420e-05 & 2.812 & 9.4263e-05 & 2.831 \\
         \tableline
         160 & 1.8127e-06 & 2.889 & 1.3420e-05 & 2.812 \\
         \tableline
         320 & 2.3663e-07 & 2.937 & 1.8127e-06 & 2.889 \\
         \tableline
         640 & 3.0264e-08 & 2.967 & 2.3663e-07 & 2.937 \\
         \tableline
    \end{tabular}
    \caption{$L^2$ error for Burgers' equation at $t = 0.5$ when employing a third-order method.
    The errors of the fine mesh and mesh filter solutions are compared. The error of the mesh-based filtering solution is computed by projecting the solution onto $\widehat{\cT}$.}
    \label{tab:Burgers_k=2_p=3}
\end{table}

\subsubsection{Burgers' Equation}
For the nonlinear Burgers' equation, we set the periodic domain as $\Omega = [-\pi,\pi]$, with initial condition $u_0(x) = 1/2 + \sin x$.
The Godunov numerical flux is used.
For this test we only evolve the solution using quadratic polynomials ($k=2$) and SSPRK3 time stepping.
The cells are merged as was done in the linear transport test. 
The time steps are
\[
	\Delta t = \mathrm{CFL}\,\min_{j}\left\{\frac{|I_j|}{\displaystyle\max_{x\in\Omega}|\bu_0(x)|}\right\}
	\quad
	\text{and}
	\quad
	\widehat{\Delta t} = \mathrm{CFL}\,\min_{p}\left\{\frac{|\widehat{I}_p|}{\displaystyle\max_{x\in\Omega}|\bu_0(x)|}\right\},
\]
where we set $\mathrm{CFL}=0.2$.
As before $\widehat{\Delta t} = 2\Delta t$.

For this initial condition, a shock forms at \(t=1\). Before shock formation, the solution remains smooth. 
Table \ref{tab:Burgers_k=2_p=3} reports the \(L^2\) error at \(t=0.5\), which is computed the same way as it was for the linear transport equation.
Like the fine-mesh solution without filtering, the mesh-based filter solution achieves third order accuracy with respect to the $L^2$ error.  
Again, we can observe from Table \ref{tab:Burgers_k=2_p=3} that the mesh-filtering solutions with $N$ cells have the same error as fine-mesh solutions with $N/2$ cells, which agrees with our conclusion in Theorem~\ref{thm:SolutionEquivalence}.
To ensure these errors are consistent with Theorem \ref{thm:SolutionEquivalence} we compute
    $\int_{I_{j}}(1/2+\sin x)\ell^{(j)}_{i}\,dx$
    and 
    $\int_{\widehat{I}_p}(1/2+\sin x)\hat{\ell}^{(p)}_{i}\,dx$
exactly when projecting the initial condition $u_0(x)$ onto the DG solution space.

To test the performance of the mesh-based filtering process when the solution becomes discontinuous, we also evolve the solution to time $t = 2$.
Figure~\ref{fig:BurgersN80} compares the fine-mesh and mesh-based filter solutions at \(t=2\) using \(N=80\) fine cells. 
A slope limiter is applied after each RK stage update, but before projecting and prolonging in the case of filtering, to control oscillations near the shock. 
The results show that the mesh-based filtering approach captures the solution through shock formation.
Visually, the solutions appear to differ only in the shock region, which is to be expected since the merging solution is resolving the shock on a coarser mesh.
This example also indicates that limiters can be combined with mesh-based filtering in practice, although, as noted above, applying a limiter independently on the fine mesh may invalidate the exact equivalence in Theorem~\ref{thm:SolutionEquivalence}.
As noted in Remark~\ref{rem:Limiters}, preserving the equivalence stated in Theorem~\ref{thm:SolutionEquivalence} in the presence of limiting requires applying the limiter on $\widehat{\cT}$ prior to prolonging to $\cT$.  
This modification may also eliminate the overshoot to the left of the shock observed in the filtered solution.
Further analysis is needed to carefully investigate the effects of combining slope or bound-enforcing limiters with the mesh-based filter.
However, in all of the numerical tests considered in this paper---including the two- and three-dimensional examples in Section~\ref{sec:Results}, which employ slope and bound-enforcing limiters---limiting does not appear to adversely affect the computed solutions.

Finally, while for these two simple tests the size of the merged cells was uniform and there was not a region where no merging occurs, we want to emphasize that this is not necessary.
In the two- and three-dimensional spherical-polar-coordinate computations presented in the next section, cells are not merged uniformly.
The number of cells merged varies with the radius (2D) or both the radius and the polar angle (3D).

\begin{figure}
    \centering
    \centering
     \begin{subfigure}[b]{0.45\textwidth}
         \centering
         \includegraphics[width=\textwidth]{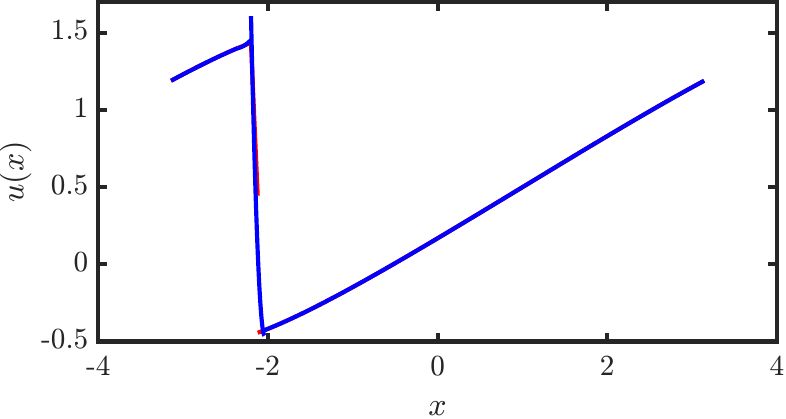}
         \label{subfig:BurgersResult}
     \end{subfigure}
     \hfill
     \begin{subfigure}[b]{0.45\textwidth}
         \centering
         \includegraphics[width=\textwidth]{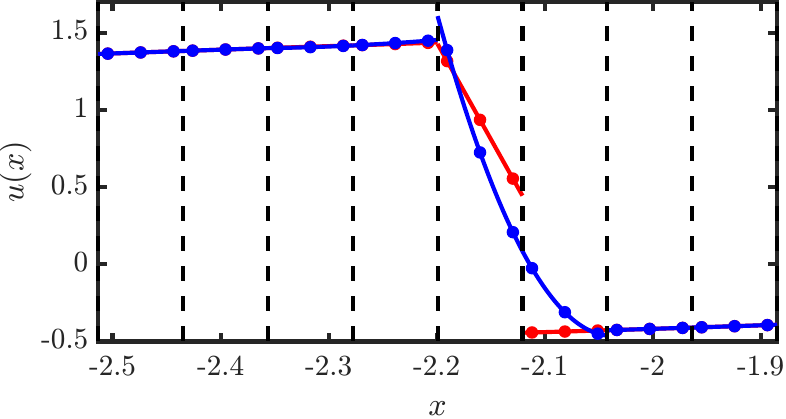}
         \label{subfig:BurgersShock}
     \end{subfigure}
    \caption{The left panel plots the solution without mesh-based filtering (red) and the solution with mesh-based filtering (blue) to the Burgers' equation at $t = 2$ using $N = 80$ cells.
    The right panel zooms in on the shock region of the solution in the left panel.
    The markers are plotted at the nodal values, and the dashed vertical lines indicate fine-cell boundaries.
    }
    \label{fig:BurgersN80}
\end{figure}

\section{Mesh-Based Filtering in Two and Three Dimensions}
\label{sec:2D3Dmerging}

In this section, we extend the one-dimensional mesh-based filtering procedure to spherical-polar meshes in two and three dimensions. 
One new ingredient is the presence of geometric volume factors in spherical coordinates.
We discuss the construction of the merged mesh $\widehat{\cT}$ and the projection--prolongation operators used in the filtering step.

To demonstrate the application of the mesh-based filter, we focus on solving the nonrelativistic Euler equations, written for general curvilinear coordinates as \citep[e.g.,][]{Pochik_2021}
\begin{subequations}
\label{eq:nonrelativisticEuler}
\begin{align}
    \p_t\rho+\frac{1}{\sqrt{\gamma}}\,\p_i(\sqrt{\gamma}\,\rho v^i) &= 0,\\
    \p_t(\rho v_j) + \frac{1}{\sqrt{\gamma}}\,\p_i(\sqrt{\gamma}\,\Pi^i_{\hspace{4pt}j}) &= \frac{1}{2}\,\Pi^{ik}\,\p_j{\gamma_{ik}},\\
    \p_tE+\frac{1}{\sqrt{\gamma}}\,\p_i(\sqrt{\gamma}\,[E+p]\,v^i) &= 0,
\end{align}
\end{subequations}
where $\rho$ is the mass density, $v^{i}$ components of the three-velocity, $\Pi^i_{\hspace{4pt}j} = \rho v^iv_j+p\,\delta^{i}_{\hspace{4pt}j}$ the stress tensor, $p$ the pressure, $E=e+\frac{1}{2}\rho v^{i}v_{i}$ the total energy density, and $e$ the internal energy density.  
We employ an ideal gas equation of state, $p=(\Gamma-1)\,e$, where $\Gamma$ is the ratio of specific heats.  
Einstein's summation is used where repeated Latin indices imply summation from $1$ to $3$.
We will use $\bu = (\rho,\rho v_j,E)$ to denote the solution vector for all evolved quantities.
The spatial metric components are denoted by $\gamma_{ij}$, and $\sqrt{\gamma}=\sqrt{\det[\gamma_{ij}]}$ is the square root of metric determinant, also referred to as the geometric volume factor.  
The spatial metric and its inverse $\gamma^{ij}$ are used to lower and raise indices on tensors, e.g., $v_{i}=\gamma_{ij}v^{j}$.  
For orthogonal curvilinear coordinates, $\gamma_{ij}$ is diagonal, and for spherical-polar coordinates
\begin{equation}
	\gamma_{11}=1, \quad \gamma_{22}=r^2, \quad \gamma_{33}=r^2\sin^2\theta, \quad \text{and} \quad \sqrt{\gamma}=r^2\sin\theta.
\end{equation}
We also introduce the scale factors $\mathsf{h}_{1}=1$, $\mathsf{h}_{2}=r$, and $\mathsf{h}_{3}=r\sin\theta$, and note that the following relations hold $\gamma_{11}=\mathsf{h}_{1}\mathsf{h}_{1}$, $\gamma_{22}=\mathsf{h}_{2}\mathsf{h}_{2}$, and $\gamma_{33}=\mathsf{h}_{3}\mathsf{h}_{3}$.  

For the nonrelativistic Euler equations, the timestep constraints are as follows in spherical-polar coordinates $(r,\theta,\phi)$
    \begin{alignat}{4}
    (\Delta t)_r &= \mathrm{CFL}\frac{\Delta r}{\lambda^{1}},\label{eq:timestep_r}\\
    (\Delta t)_\theta &= \mathrm{CFL}\frac{r\Delta\theta}{\lambda^{2}}&&\approx(\Delta t)_r\Delta\theta \quad &&\text{if} \quad &&r\approx\Delta r,\label{eq:timestep_theta}\\
    (\Delta t)_\phi &= \mathrm{CFL}\frac{r\sin\theta\Delta\phi}{\lambda^{3}}&&\approx(\Delta t)_r\Delta\theta\Delta\phi \quad &&\text{if} \quad &&r\approx\Delta r \quad \text{and} \quad \sin\theta\approx\Delta\theta.
    \label{eq:timestep_phi}
\end{alignat}
Here \(\lambda^{i}\) denote representative upper bounds on the relevant physical wave speeds. 
The timestep $\Delta t$ is then taken as
\begin{equation}
\Delta t = \min\{(\Delta t)_r, (\Delta t)_\theta, (\Delta t)_\phi\}.
\end{equation}
Near the origin and, in three dimensions, near the polar axis, the time step sizes \((\Delta t)_\theta\) and \((\Delta t)_\phi\) can be much smaller than \((\Delta t)_r\) since $\Delta\theta\ll 1$ and $\Delta\phi\ll 1$.
The goal of mesh-based filtering is to merge cells, first in the $\theta$-direction and then in the $\phi$-direction, so that $(\Delta t)_\theta$ and $(\Delta t)_\phi$ become similar in size to $(\Delta t)_r$.
Then it follows that $\Delta t \approx (\Delta t)_r$.

\subsection{Two-Dimensional Merging}
In two dimensions, we consider an $N_r\times N_\theta$ mesh, where $N_r$ and $N_\theta$ are the number of cells in the $r$-direction and $\theta$-direction respectively.
We let $r\in[0,R]$ and $\theta\in[0,\pi]$.
For simplicity, we assume the cells are uniformly spaced in each dimension, so that $\Delta r = R/N_r$ and $\Delta\theta = \pi/N_\theta$.
Then
\[
    \cT =
    \left\{
    I_i\times J_j: i=1,\ldots,N_r,\quad j=1,\ldots,N_\theta
    \right\}.
\]
On each element \(I_i\times J_j\), we express the numerical solution as
\begin{equation}
    \bu_h(t,r,\theta)\big|_{I_i\times J_j} =
    \sum_{\alpha=1}^{k+1}\sum_{\beta=1}^{k+1}u^{(ij)}_{\alpha\beta}(t)\ell^{(i)}_{\alpha}(r)\ell^{(j)}_{\beta}(\theta).
\end{equation}

As to the merged mesh, for each $i$, the $\theta$-direction of the strip $\cup_{j=1}^{N_{\theta}}I_i\times J_j$ will be discretized into $1\leq\hat{N}_{\theta}(i)\leq N_{\theta}$ merged cells.
We will denote merged cells in the $\theta$-direction as
\[
	\hat{J}_p(i) = \bigcup_{m\in B_{p}(i)}J_{m},
\]
where for each $p$ there is an indexed set $B_p(i)\subset\{1,\ldots,N_{\theta}\}$ which collects the fine cells $J_m$ comprising $\hat{J}_p(i)$.
The cells $J_{m}$ must be contiguous in the $\theta$-direction of the strip.
We define $|B_p(i)| := M_\theta(i,p)$ as the number of fine angular cells comprising the $p$-th merged cell at the level of $I_i$.
In general, the only requirement is that for each $i$,
\begin{equation}\label{eq:sumMtheta}
    \sum_{p=1}^{\hat{N}_{\theta}} M_\theta(i,p) = N_\theta.
\end{equation}
The two-dimensional merged mesh is then,
\[
    \widehat{\cT} =
    \left\{
    I_i\times\widehat J_p(i): i=1,\ldots,N_r,\quad p=1,\ldots,\widehat N_\theta(i)
    \right\}.
\]
On a merged element \(I_i\times\widehat J_p(i)\), we express the numerical solution as
\begin{equation}
    \hat{\bu}_h(t,r,\theta) \big|_{I_i\times\widehat J_p(i)} =
    \displaystyle\sum_{\alpha=1}^{k+1}\sum_{\beta=1}^{k+1}\hat{u}^{(ip)}_{\alpha\beta}(t)\ell^{(i)}_{\alpha}(r)\hat{\ell}^{(p)}_{\beta}(\theta),
\end{equation}
where $\widehat\ell_\beta^{(i,p)}$ is the nodal basis function on $\widehat J_p(i)$.
We will now discuss how we determine $M_\theta(i,p)$ and the cells $J_{m}$ comprising a merged cell.

\subsubsection{$\theta$-Direction Merging Criterion}

If we look at the timestep determined by $\cT$, then for $r<\Delta r/\Delta\theta$, Equation~\eqref{eq:timestep_theta} will result in the more restrictive timestep, while for $r>\Delta r/\Delta\theta$ Equation~\eqref{eq:timestep_r} is the more restrictive timestep.
To then avoid having Equation~\eqref{eq:timestep_theta} determine the timestep, we choose the angular merge factor so that the merged angular length is comparable to the radial length.

For each $I_i$, we merge cells in the $\theta$-direction in the following way
\begin{equation}
\label{eq:theta_merge_criteria}
    \Delta\hat{\theta}(i,p)=M_\theta(i,p)\Delta\theta,
    \quad
    \text{where}
    \quad
    M_\theta(i,p)=\argmin_{\bbN}\left(rM_\theta(i,p)\Delta\theta>\Delta r, \quad \forall r\in I_i\right).
\end{equation}
If \(M_\theta(i,p)=1\), no angular merging is performed at the level $I_i$.
For simplicity we force $N_r$ and $N_\theta$ to be powers of $2$, and for each $i$, we also force $M_\theta(i,p)$ to be constant with respect to $p$, in order that merged cells are uniform at each level of $I_i$.
By Equation~\eqref{eq:sumMtheta}, $N_\theta = \hat{N}_\theta M_{\theta}(i,p)$, which indicates $\hat{N}_{\theta}$ and $M_\theta(i,p)$ are also powers of $2$.
Note that the new cell width in the $\theta$-direction, $\Delta\hat{\theta}(i,p)$, depends on $r$, and decreases as $r$ increases. 
For $r>\Delta r/\Delta\theta$, we have $\Delta\hat{\theta}(i,p)=\Delta\theta$.
The top row of Figure \ref{fig:2DCoordinateGrid} depicts the result of Equation~\eqref{eq:theta_merge_criteria} when applied to a $32\times 32$ mesh.
Its left panel shows the result in coordinate space, while the right panel maps the coordinate space into the physical space.

\subsubsection{Mesh-Based Filter in Two Dimensions}
The appearance of geometry terms in Equation~\eqref{eq:nonrelativisticEuler} requires us to make a slight modification to the projection operator we introduced in Section \ref{sec:1DCellMerging}.
Preliminary testing of the mesh-based filter demonstrated it was necessary to include $\sqrt{\gamma}$ in the projection operator for the purposes of mass and energy conservation.
The projection, $\cP$, onto the fine element $I_i\times J_j\in\cT$ is now given by
\begin{equation}
\label{eq:2Dprolongation}
    2\pi\int_{I_i}\int_{J_j}\cP[u]\ell_\alpha(r)\ell_\beta(\theta)\sqrt{\gamma}\,dr\,d\theta
    =
    2\pi\int_{I_i}\int_{J_j}u\ell_\alpha(r)\ell_\beta(\theta)\sqrt{\gamma}\,dr\,d\theta,
\end{equation}
for all \(\alpha,\beta=1,\ldots,k+1\).

On each merged element, if the exact metric factor \(\sqrt{\gamma}\) were used with the exact integration, the merged projection could be defined using \(\sqrt{\gamma}\) on both sides. In the implementation with quadrature rules, however, metric quantities are represented on the nodes on which the solution is represented.
The nodal values of $\sqrt{\gamma}$ are computed using polynomial representations of $\mathsf{h}_1, \mathsf{h}_2$, and $\mathsf{h}_3$ as described in Section 3.1.1 of \citet{Pochik_2021}.
We therefore define a merged-mesh representation \(\sqrt{\widehat\gamma}\) by matching its moments against the merged basis functions.
On a merged element \(I_i\times\widehat J_p(i)\), we define \(\sqrt{\widehat\gamma}\in\bbV^{k}_{h}\) by
\begin{equation}
\label{eq:2Dprojection_gamma}
    2\pi\int_{I_i}\int_{\hat{J}_p(i)}\ell_\alpha(r)\hat{\ell}_\beta(\theta)\sqrt{\widehat{\gamma}}\,dr\,d\theta
    = 2\pi\int_{I_i}\int_{\hat{J}_p(i)}\ell_\alpha(r)\hat{\ell}_\beta(\theta)\sqrt{\gamma}\,dr\,d\theta.
\end{equation}
The nodal values for $\sqrt{\widehat{\gamma}}$ on the mesh $\widehat{\cT}$ can be computed using Equation~\eqref{eq:localMergedNodalValue} and replacing $\widehat{u}$ and $u$ respectively with $\sqrt{\widehat{\gamma}}$ and $\sqrt{\gamma}$.
Then the projection, $\widehat{\cP}$, onto the merged element $I_i\times\hat{J}_p(i)\in\widehat{\cT}$ is given by
\begin{equation}
\label{eq:2Dprojection}
    2\pi\int_{I_i}\int_{\hat{J}_p(i)}\widehat{\cP}[u]\ell_\alpha(r)\hat{\ell}_\beta(\theta)\sqrt{\widehat{\gamma}}\,dr\,d\theta
    =    2\pi\int_{I_i}\int_{\hat{J}_p(i)}u\ell_\alpha(r)\hat{\ell}_\beta(\theta)\sqrt{\gamma}\,dr\,d\theta.
\end{equation}

Then as in one-dimension, for any fine element $I_i\times J_j\subset I_i\times\hat{J}_p(i)$, the projection--prolongation updated solution of $\bu_h$ on that element is given by
\begin{equation}\label{eq:ProjectionProlongationUpdate2D}
    \bu_h\Big|_{I_i\times J_j}
    =
    \widehat{\cP}[\bu_h]\Big|_{I_i\times J_j}.
\end{equation}
Again, the idea is to project the solution onto $\widehat{\cT}$ and follow this with a prolongation onto $\cT$.
The presence of $\sqrt{\hat{\gamma}}$ in the projection slightly complicates the details of the mesh-based filtering process.
In practice, since $\sqrt{\gamma}$ does not change in time for the nonrelativistic Euler equations, $\sqrt{\hat{\gamma}}$ can easily be computed and stored after $\widehat{\cT}$ is determined.

\begin{figure}
     \centering
     \begin{subfigure}[b]{0.45\textwidth}
         \centering
         \includegraphics[width=\textwidth]{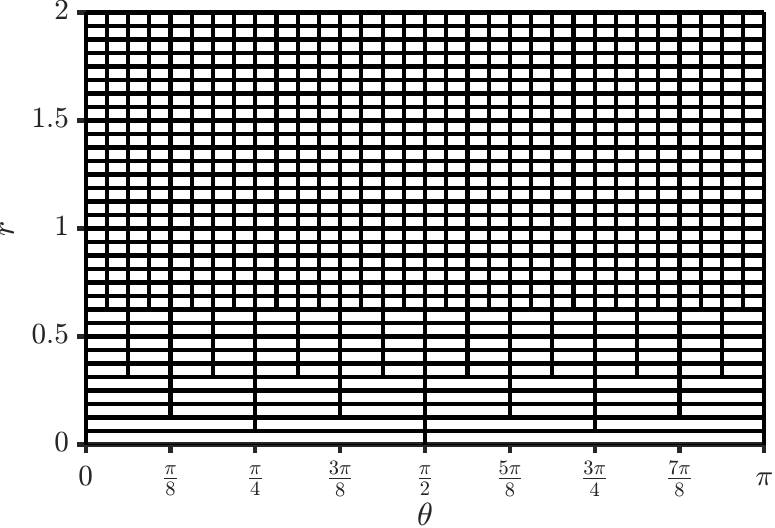}
     \end{subfigure}
     \begin{subfigure}[b]{0.45\textwidth}
         \centering
         \includegraphics[scale=0.75]{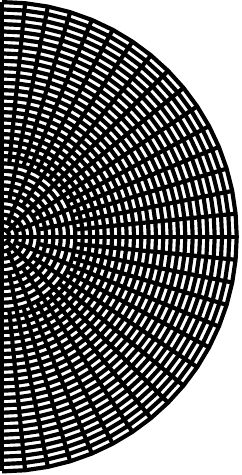}
     \end{subfigure}
     \begin{subfigure}[b]{0.23\linewidth}
        \centering
        \includegraphics[width=\textwidth]{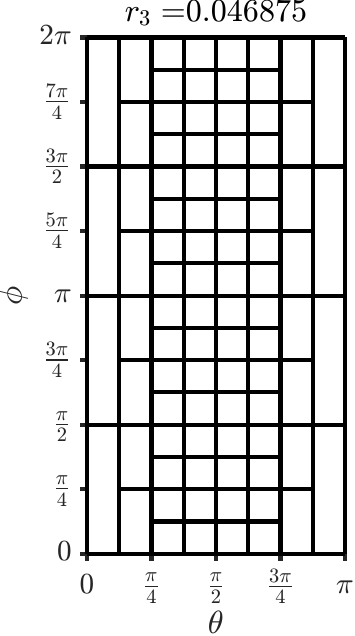}
    \end{subfigure}
    \hfill
    \begin{subfigure}[b]{0.23\linewidth}
        \centering
        \includegraphics[width=\textwidth]{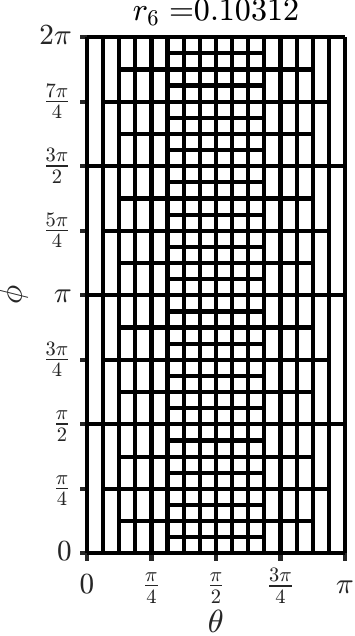}        
    \end{subfigure}
    \hfill
    \begin{subfigure}[b]{0.23\linewidth}
        \centering
        \includegraphics[width=\textwidth]{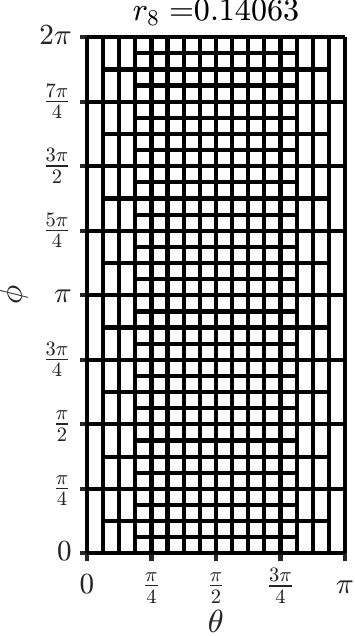}        
    \end{subfigure}
    \hfill
    \begin{subfigure}[b]{0.23\linewidth}
        \centering
        \includegraphics[width=\textwidth]{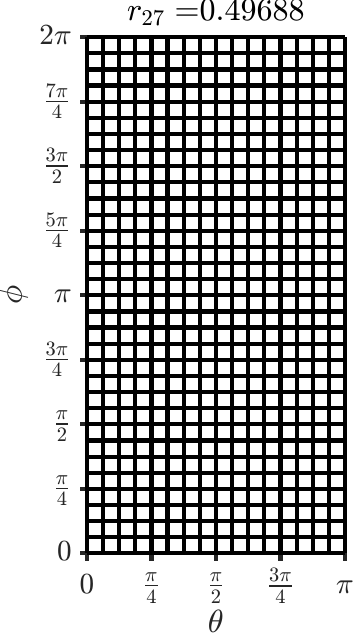}        
    \end{subfigure}
        \caption{Top row: the result of applying Equation~\eqref{eq:theta_merge_criteria} to a $32\times 32$ mesh in coordinate space (left) and the same mesh transformed into the physical space (right).
        Bottom row: depicted are $\theta$-$\phi$ coordinate meshes for four different slices of $r$ for a $64\times16\times 32$ mesh.
        The slices of $r$ from left to right are from the third, sixth, eighth, and twenty-seventh radial shells.
        }
        \label{fig:2DCoordinateGrid}
\end{figure}

\subsection{Three-Dimensional Merging}
We now extend the construction to three dimensions.
Consider an $N_r\times N_\theta\times N_\phi$ mesh, where $N_r$, $N_\theta$, and $N_\phi$ are the number of cells in the $r$-,$\theta$-, and $\phi$-directions, respectively.
We let $r\in[0,R]$, $\theta\in[0,\pi]$, and $\phi\in[0,2\pi]$.
For simplicity, we again assume uniform spacing for cells, so that $\Delta r = R/N_r$, $\Delta\theta = \pi/N_\theta$, and $\Delta\phi = 2\pi/N_\phi$.
Then
\[
    \cT = \left\{
    I_i\times J_j\times K_l: i=1,\ldots,N_r,\quad j=1,\ldots,N_\theta,\quad l=1,\ldots,N_\phi
    \right\}.
\]
On each element \(I_i\times J_j \times K_l\), we express the numerical solution as
\begin{equation}
    \bu_h(r,\theta,\phi)|_{I_i\times J_j\times K_l} =
    \sum_{\alpha=1}^{k+1}\sum_{\beta=1}^{k+1}\sum_{\delta=1}^{k+1}u^{(ijl)}_{\alpha\beta\delta}(t)\ell^{(i)}_{\alpha}(r)\ell^{(j)}_{\beta}(\theta)\ell^{(l)}_{\delta}(\phi).
\end{equation}

As to the merged mesh, we first merge in the $\theta$-direction as is done in two dimensions.
This gives the intermediate mesh
\[
	\bigcup_{i=1}^{N_r}\bigcup_{p=1}^{\hat{N}_\theta(i)}\bigcup_{l=1}^{N_\phi}I_i\times\hat{J}_p(i)\times K_l.
\]
Merging in the $\phi$-direction is then applied to this mesh to arrive at $\widehat{\cT}$.
For each pair $(i,p)$, the $\phi$-direction of the strip $\cup_{i=1}^{N_r}\cup_{p=1}^{\hat{N}_{\theta}(i)}I_i\times \hat{J}_{p}(i)\times K_l$ will be discretized into $1\leq \hat{N}_{\phi}(i,p)\leq N_{\phi}$ merged cells.
We will denote the merged cells in the $\phi$-direction as
\[
	\hat{K}_q(i,p) = \bigcup_{n\in C_q(i,p)}K_{n},
\]
where for each pair $(i,p)$ there is an indexed set $C_q(i,p)\subset\{1,\ldots,N_{\phi}\}$ which collects the fine cells $K_n$ comprising $\hat{K}_q(i,p)$.
The cells $K_{n}$ must be contiguous in the $\phi$-direction of the strip.
We define $|C_q(i,p)| := M_{\phi}(i,p,q)$ as the number of fine cells comprising the $q$-th merged cell at the level of $I_i\times\hat{J}_{p}(i)$.
In general, the only requirement is that for each pair $(i,p)$,
\begin{equation}\label{eq:sumMphi}
    \sum_{q=1}^{\hat{N}_{\phi}(i,p)}M_{\phi}(i,p,q) = N_{\phi}.
\end{equation}
The three-dimensional merged mesh is then,
\[
    \widehat{\cT} = \left\{ I_i\times\widehat J_p(i)\times\widehat K_q(i,p): i=1,\ldots,N_r,\quad p=1,\ldots,\widehat N_\theta(i),\quad q=1,\ldots,\widehat N_\phi(i,p)
    \right\}.
\]
On a merged element $I_i\times\hat{J}_{p}(i)\times\hat{K}_{q}(i,p)$, we express the numerical solution as
\begin{equation}
    \hat{\bu}_h(r,\theta,\phi) \big|_{I_i\times\widehat J_p(i)\times\widehat K_q(i,p)} =
    \displaystyle\sum_{\alpha=1}^{k+1}\sum_{\beta=1}^{k+1}\sum_{\delta=1}^{k+1}\hat{u}^{(ipq)}_{\alpha\beta\delta}(t)\ell^{(i)}_{\alpha}(r)\hat{\ell}^{(p)}_{\beta}(\theta)\hat{\ell}^{(q)}_{\delta}(\phi).
\end{equation}
We will now discuss how we determine $M_{\phi}(i,p,q)$ and the cells $K_{n}$ comprising a merged cell.

\subsubsection{$\phi$-Direction Merging Criterion}
To avoid the restrictive timestep from the $\theta$-dimension we first merge in the $\theta$-direction using the criterion in Equation~\eqref{eq:theta_merge_criteria}.
After the $\theta$-direction merging has been determined, the $\phi$-direction time-step restriction is governed by $r\sin\theta\,\Delta\phi$. 
To then avoid having Equation~\eqref{eq:timestep_phi} determine the timestep, we choose the $\phi$-merge factor so that the merged $\phi$-direction length is comparable to the radial length.

For each $I_i\times\hat{J}_{p}(i)$, we use the following merging criterion for the $\phi$-direction
\begin{equation}
    \label{eq:phi_merge_criteria}
    \Delta\hat{\phi}(i,p,q)=M_{\phi}(i,p,q)\Delta\phi,
    \quad
    \text{where}
    \quad
    M_{\phi}(i,p,q)=\argmin_{\bbN}\left(r\sin\hat{\theta}M_{\phi}(i,p,q)\Delta\phi>\Delta r, \quad \forall r\in I_i, \hat{\theta}\in\hat{J}_{p}(i) \right).
\end{equation}
Here we have used $\hat{\theta}$ to indicate that the merging in the $\phi$-direction depends on the merging done in the $\theta$-direction.
In other words, $\sin\hat{\theta}$ should be evaluated based on merged cell widths in the $\theta$-direction.
For simplicity we again force $N_r$, $N_\theta$, and $N_\phi$ to be powers of $2$, and for each pair $(i,p)$, we also force $M_{\phi}(i,p,q)$ to be constant with respect to $q$, so that merged cells are uniform at each level of $I_i\times\hat{J}_{p}(i)$.
Then as a result of Equation~\eqref{eq:sumMphi}, $M_{\phi}(i,p,q)$ is also a power of $2$.
The bottom row of Figure \ref{fig:2DCoordinateGrid} depicts the $\theta$-$\phi$ coordinate space at select radial shells after applying Equations~\eqref{eq:theta_merge_criteria} and \eqref{eq:phi_merge_criteria} to a $64\times 16\times 32$ mesh.
The four radial shells depicted from left to right are the third, sixth, eighth, and twenty-seventh.
To explain how to read the figure, let us look at the leftmost mesh depicting the third radial shell.
The $16$ cells in the $\theta$-direction have been merged into $8$ cells according to Equation~\eqref{eq:theta_merge_criteria}.
Each of these eight merged cells merges independently in the $\phi$-direction according to Equation~\eqref{eq:phi_merge_criteria}:
the first cell merges the $\phi$-direction into $4$ cells; the second cell, into $8$ cells; the third cell, into $16$ cells, etc.

\subsubsection{Projection and Prolongation Operators in Three Dimensions}
The three-dimensional projection operators are the direct analogues of the two-dimensional ones. 
The projection, $\cP$, onto the fine element $I_i\times J_j\times K_l\in\cT$ is now given by
\begin{equation}
\label{eq:3Dprolongation}
    \int_{I_i}\int_{J_j}\int_{K_l}\cP[u]\ell_\alpha(r)\ell_\beta(\theta)\ell_\delta(\phi)\sqrt{\gamma}\,dr\,d\theta\,d\phi
    =
    \int_{I_i}\int_{J_j}\int_{K_l}u\ell_\alpha(r)\ell_\beta(\theta)\ell_\delta(\phi)\sqrt{\gamma}\,dr\,d\theta\,d\phi,
\end{equation}
for all \(\alpha,\beta,\delta=1,\ldots,k+1\).

On a merged element $I_i\times\hat{J}_p(i)\times\hat{K}_q(i,p)$, we first define the merged-mesh representation $\sqrt{\widehat\gamma}\in\bbV^{k}_{h}$ by
\begin{equation}
\label{eq:3Dprojection_gamma}
    \int_{I_i}\int_{\hat{J}_p(i)}\int_{\hat{K}_{q}(i,p)}\ell_\alpha(r)\hat{\ell}_\beta(\theta)\hat{\ell}_\delta(\phi)\sqrt{\hat{\gamma}}\,dr\,d\theta\,d\phi
    =
    \int_{I_i}\int_{\hat{J}_p(i)}\int_{\hat{K}_{q}(i,p)}\ell_\alpha(r)\hat{\ell}_\beta(\theta)\hat{\ell}_\delta(\phi)\sqrt{\gamma}\,dr\,d\theta\,d\phi.
\end{equation}
The nodal values of $\sqrt{\widehat{\gamma}}$ on the mesh $\widehat{\cT}$ are computed by extending Equation~\eqref{eq:localMergedNodalValue} to the setting where merging occurs in two dimensions.
Then the projection, $\widehat{\cP}$, onto the merged element $I_i\times\hat{J}_p(i)\times\hat{K}_q(i,p)$ is given by
\begin{equation}
\label{eq:3Dprojection}
    \int_{I_i}\int_{\hat{J}_p(i)}\int_{\hat{K}_{q}(i,p)}\widehat{\cP}[u]\ell_\alpha(r)\hat{\ell}_\beta(\theta)\hat{\ell}_\delta(\phi)\sqrt{\hat{\gamma}}\,dr\,d\theta\,d\phi
    =
    \int_{I_i}\int_{\hat{J}_p(i)}\int_{\hat{K}_{q}(i,p)}u\ell_\alpha(r)\hat{\ell}_\beta(\theta)\hat{\ell}_\delta(\phi)\sqrt{\gamma}\,dr\,d\theta\,d\phi.
\end{equation}
Then, for any element $I_i\times J_j\times K_l\subset I_i\times\hat{J}_p(i)\times\hat{K}_q(i,p)$ the projection--prolongation updated solution of $\bu$ on that element is given by
\begin{equation}\label{eq:ProjectionProlongationUpdate3D}
    \bu_h\Big|_{I_i\times J_j\times K_l}
    =
    \widehat{\cP}[\bu_h]\Big|_{I_i\times J_j\times K_l}.
\end{equation}
As in two dimensions, $\sqrt{\hat{\gamma}}$ is time independent, and is computed once through Equation~\eqref{eq:3Dprojection_gamma} after $\widehat{\cT}$ is determined.

\section{Numerical Tests}
\label{sec:Results}

The mesh-based filter has been implemented in the toolkit for high-order neutrino-radiation hydrodynamics \citep[\texttt{thornado};][]{endeve_etal_2019,Pochik_2021}.
The implementation supports two- and three-dimensional spherical-polar meshes and uses the projection--prolongation operators defined in Equations~\eqref{eq:ProjectionProlongationUpdate2D} and \eqref{eq:ProjectionProlongationUpdate3D}.  
Following the flowchart in Figure \ref{fig:MergingFlowchart}, the filter is applied to the initial condition and after every stage of the RK time integrator.
In addition, a slope limiter and a bound-enforcing limiter are used.  
We apply the slope limiter on $\cT$ before the mesh-based filter, allowing us to retain \texttt{thornado}'s existing structured-mesh implementation without modification.  
This ordering falls outside the conditions for the equivalence established in Theorem~\ref{thm:SolutionEquivalence}; consequently, stability properties of the RKDG scheme on the merged mesh, such as the TVD property \citep{cockburnShu_2001}, are not guaranteed to carry over to the filtered scheme.  
However, we observed no adverse effects in the numerical experiments considered here.  
Applying the slope limiter on $\widehat{\cT}$ would require modifying the limiter routine to access cell averages in neighboring merged elements.  
The bound-enforcing limiter is applied on $\widehat{\cT}$ before prolongation to $\cT$.
On each merged element, positivity is enforced at the union of nodes from its constituent fine-mesh elements, which we found necessary to avoid numerical failures.  
Because the limiter is element local, the only modification to \texttt{thornado}'s existing implementation was to allow the set of positivity-enforcement nodes to be supplied as input to the limiter.  
We have tested the mesh-based filter in \texttt{thornado} with the following three numerical tests to assess the accuracy, robustness, and time step improvements.  

\subsection{Two-dimensional Riemann Problem}

Here we consider a two-dimensional Riemann problem, aiming to quantify the tradeoff between effective angular resolution and time step size that comes with the mesh-based filter.  
We modify the initial condition of the Sod shock tube problem \citep{Sod_1978}---see, e.g., \citet{omang_etal_2006}, for numerical solutions in spherical symmetry---by introducing $\theta$-dependence so that the solution develops non-uniformly in $\theta$ and to observe associated errors due to  the mesh-based filter, which has no negligible effect on the solution in spherical symmetry.

The domain is $\Omega = [0,2]\times[0,\pi]$, and the initial conditions are
\begin{align}
        \rho(r,\theta) &= \begin{cases}
            1+0.5\sin^2(\theta) \quad &r\leq0.4,\\
            0.125 \quad &r>0.4,
        \end{cases}\qquad
        v_r(r,\theta)  = 
        v_\theta(r,\theta) = 0,\qquad
        p(r,\theta) = \begin{cases}
            1+0.5\sin^2(\theta) \quad &r\leq 0.4,\\
            0.1 \quad &r>0.4.
        \end{cases}
\end{align}
We impose reflecting boundary conditions at the inner and outer boundaries in both dimensions.  
Using linear elements, we employ four different meshes to quantify the error introduced by mesh-based filtering.  
Three meshes, with resolutions $N_{r}\times N_{\theta}=128\times 16$, $128\times 32$, and $128\times 64$, are merged according to Equation~\eqref{eq:theta_merge_criteria}, and the fourth mesh is a relaxed $128\times 64$-mesh, where the cells are merged sub-optimally with respect to the time step, following the modified criteria
\begin{equation}
    \label{eq:theta_merge_criteria_relaxed}
    \Delta\hat{\theta}(i,p)=M_{\theta}(i,p)\Delta\theta,
    \quad
    \text{where}
    \quad
    M_{\theta}(i,p)=\argmin_{\bbN}\left(rM_{\theta}(i,p)\Delta\theta>\frac{\Delta r}{4}\right).
\end{equation}  
The merging region, $\Omega_2$, for all four meshes is depicted in Figure~\ref{fig:linearRiemannX2DependenceGrids}: Elements for which two or more fine cells have been merged are shaded.  
The effect of the relaxed merging criterion can be seen by comparing the two rightmost panels of the figure.  
The second to last panel is a $128\times 64$ mesh which merges cells according to Equation~\eqref{eq:theta_merge_criteria}, while the last panel is a $128\times 64$ mesh which merges cells according to Equation~\eqref{eq:theta_merge_criteria_relaxed}.
This relaxed criterion for merging results in only $3$ levels of merging in the $\theta$-dimension: one layer of $8$ cells, one layer of $16$ cells, three layers of $32$ cells, and the remaining layers are the original mesh with $64$ cells in the $\theta$ dimension.  
This makes the rightmost mesh structurally similar to the leftmost $128\times 16$ merged mesh.  
We choose the first three meshes to demonstrate that with mesh-based filtering the time step remains constant as the angular resolution is refined, while the last mesh is used to demonstrate the tradeoff between solution accuracy and time step size that comes with filtering.

\begin{figure}
    \centering
    \begin{subfigure}[b]{0.23\linewidth}
        \centering
        \includegraphics[width=\textwidth,trim=46 5 70 15,clip]{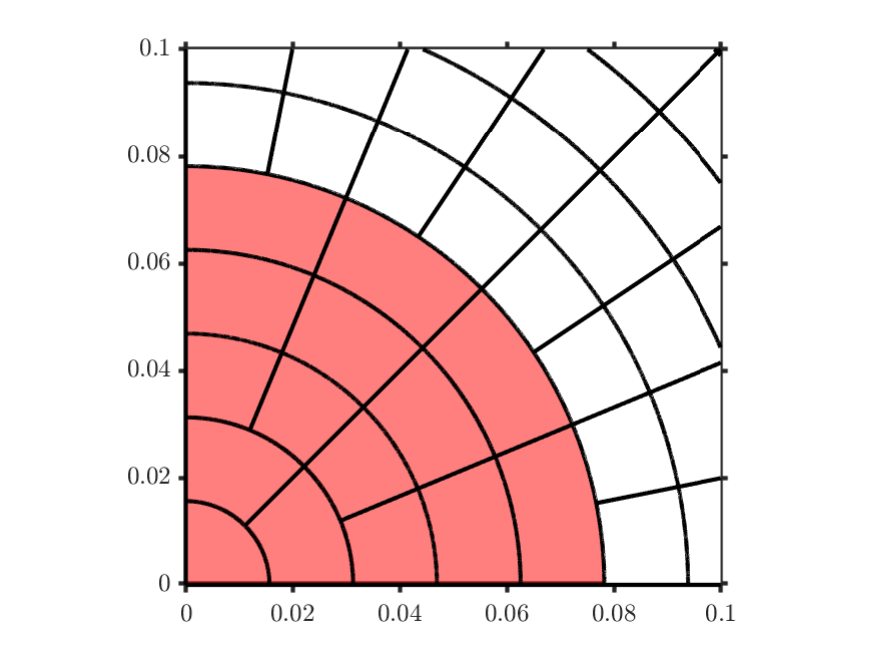}
    \end{subfigure}
    \hfill
    \begin{subfigure}[b]{0.23\linewidth}
        \centering
        \includegraphics[width=\textwidth,trim=55 5 70 15,clip]{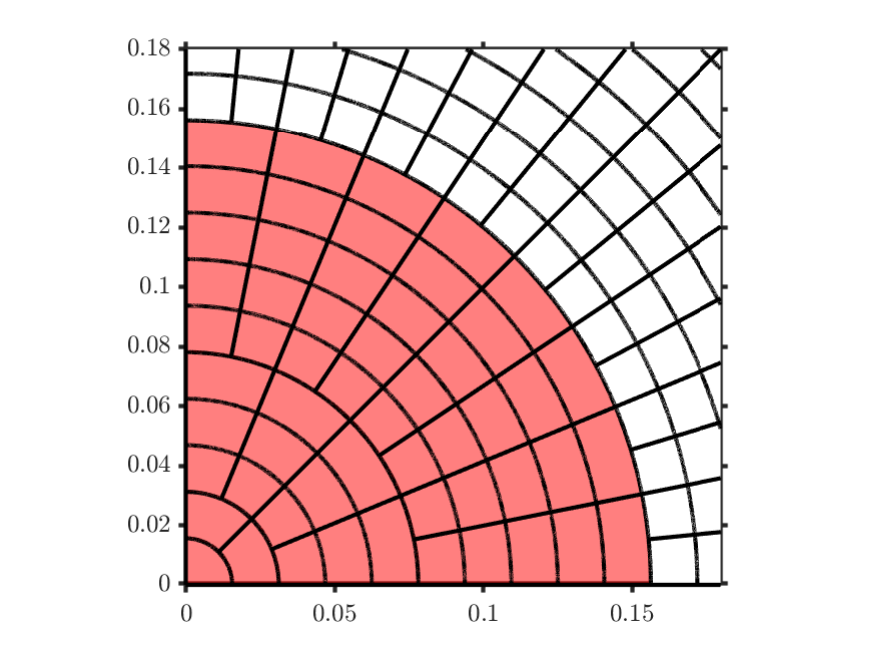}        
    \end{subfigure}
    \hfill
    \begin{subfigure}[b]{0.23\linewidth}
        \centering
        \includegraphics[width=\textwidth,trim=55 5 70 15,clip]{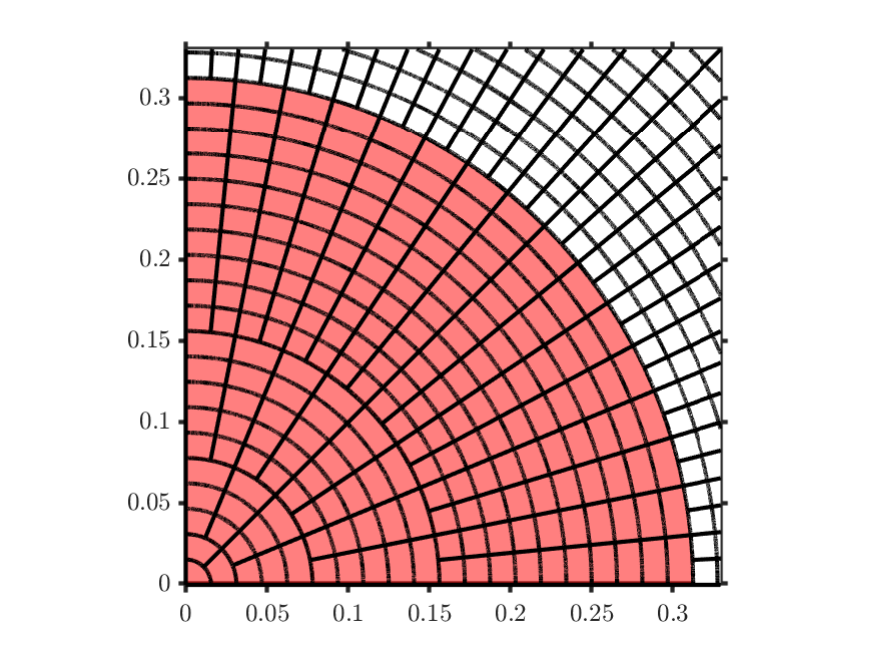}        
    \end{subfigure}
    \hfill
    \begin{subfigure}[b]{0.23\linewidth}
        \centering
        \includegraphics[width=\textwidth,trim=55 5 70 15,clip]{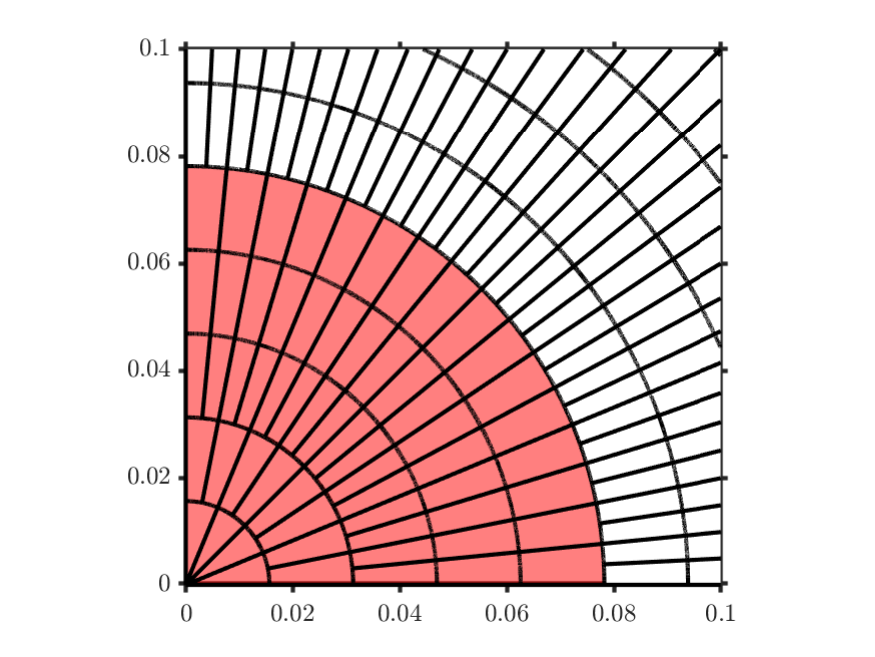}        
    \end{subfigure}
    \caption{The upper half ($0\leq\theta\leq\pi/2$) of the merging region, $\Omega_2$, of all four $\widehat{\cT}$ meshes used for the two-dimensional Riemann problem.
    Elements in which two or more fine cells have been merged together are shaded.  
    The first three meshes are merged according to Equation~\eqref{eq:theta_merge_criteria}, while the rightmost mesh is merged according to Equation~\eqref{eq:theta_merge_criteria_relaxed}.}
    \label{fig:linearRiemannX2DependenceGrids}
\end{figure}

Figure~\ref{fig:RS_128x64allgrids_video} shows radial profiles of mass density, pressure, radial velocity, and angular velocity at $t=2.5$ for the unmerged simulation (first row), the simulation merged according to Equation~\eqref{eq:theta_merge_criteria} (second row), and the simulation merged according to Equation~\eqref{eq:theta_merge_criteria_relaxed} (third row).
The animation of this figure (included with the online article) shows the evolution of the radial profiles.
The animation reveals how the $\theta$-dependence of the solutions develop over time.  
There is good agreement among the models outside the merging regions.  
The angular dependence becomes strong near the origin as the simulation approaches $t=2.5$.
The animation also shows the effects of waves reflecting off the inner and outer boundaries.
Reflections from the inner boundary occur at $t=0.3$ and $t=1.275$, while a reflection from the outer boundary occurs at $t=1.125$.  
At approximately $t=2.075$, the wave reflected from the inner boundary at $t=1.275$ collides with the wave reflected off the outer boundary at $t=1.125$.
At $t=2.5$, comparison of the unmerged simulation (first row) with the simulation merged according to Equation~\eqref{eq:theta_merge_criteria} (second row) shows noticeable differences in the angular variations of the mass density, radial velocity, and angular velocity within the merging region near the origin.
The simulation employing the relaxed merging criterion in Equation~\eqref{eq:theta_merge_criteria_relaxed} (third row) is in closer agreement with the unmerged simulation near the origin.

\begin{figure}
\begin{interactive}{animation}{RS_128x64_allgrids.mp4}
\includegraphics[width = \textwidth]{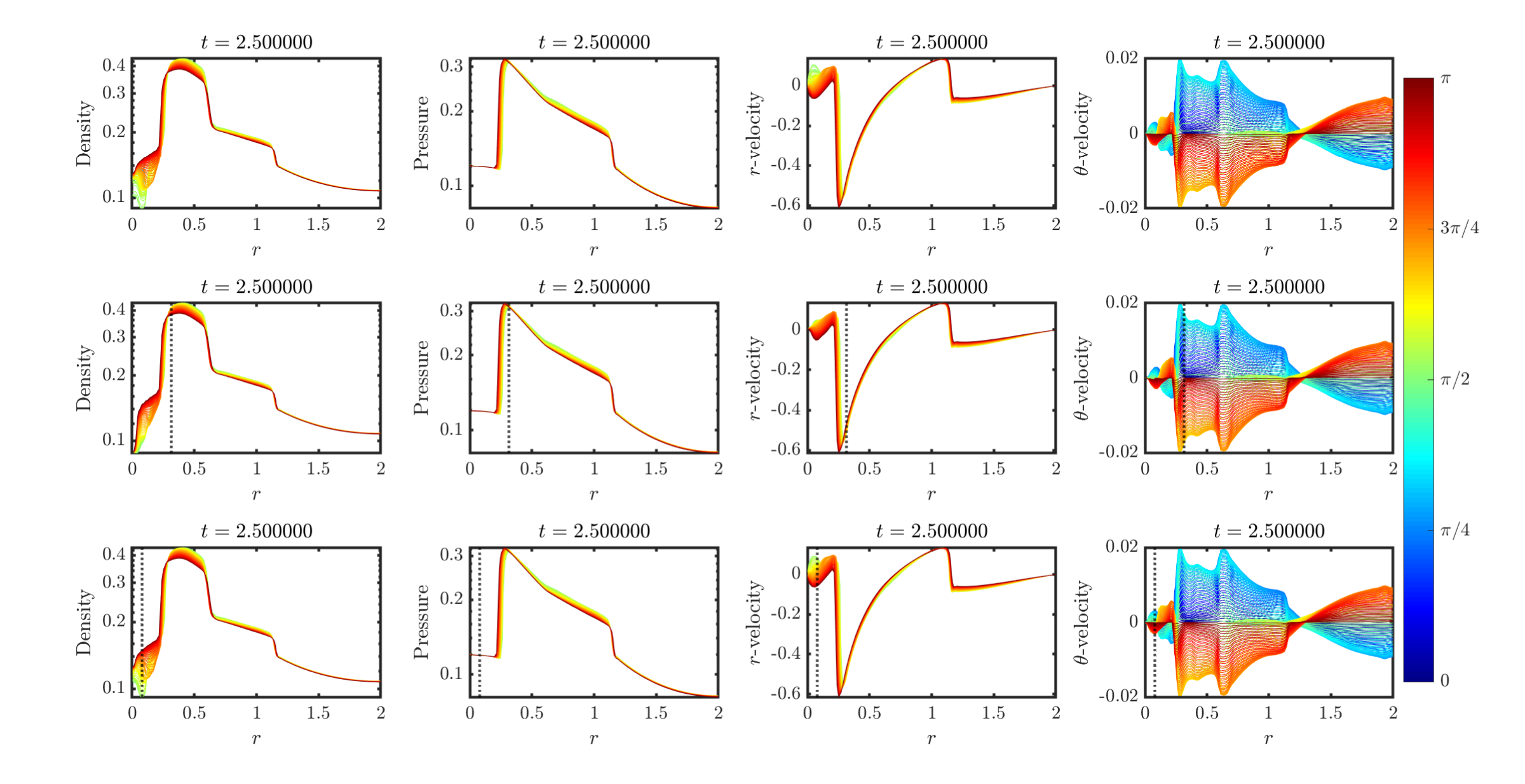}
\end{interactive}
\caption{
The four panels on each row plot radial lineouts of the density, pressure, $\theta$-velocity, and $r$-velocity for the two-dimensional Riemann problem.
The first row is the $128\times 64$ fine mesh solution, the second row is the $128\times 64$ mesh-filtering solution, and the last row is the $128\times 64$ mesh-filtering solution with relaxed cell merging.
The color scheme is used to indicate the value of $\theta$ for a radial lineout.
It ranges across the color spectrum from blue to red, where the darkest blue is the radial lineout for $\theta = 0.010373$ and the darkest red is for $\theta = 3.1312$.
The black dashed line indicates the boundary between merging and non-merging region of the mesh, approximately $r=0.08$ for the mesh-filtering solution, and approximately $r=0.315$ for the mesh-filtering solution with relaxed cell merging.
The animation of this figure ($10$~s; $t=0$ to $t=2.5$) illustrates the $\theta$-dependence of this solution.
The static figure corresponds to the solution at $t=2.5$.
}
\label{fig:RS_128x64allgrids_video}
\end{figure}

We compute the error on each mesh against an unfiltered $128\times 128$ reference solution, which we denote by $\bu_{\mathrm{ref}}$.
The error is computed using the following weighted $L^p$ norm defined over a region $R\subseteq\Omega$
\begin{equation}
    \|\cdot\|_{L^p_{\gamma}(R)} = \left(\frac{1}{|R|_\gamma}\iint\limits_{R}|\cdot|^p\sqrt{\gamma}\,dr\,d\theta\right)^{1/p}, 
    \qquad
    |R|_\gamma=\iint\limits_{R}\sqrt{\gamma}\,dr\,d\theta.
\end{equation}
Figure~\ref{fig:linearRiemannX2DependenceError} depicts $\|\rho_{\mathrm{ref}}-\rho_h\|_{L^p_{\gamma}(R)}$ (solid lines) and $\|\rho_{\mathrm{ref}}-\widehat{\rho}_h\|_{L^p_{\gamma}(R)}$ (dashed lines) versus time for all four meshes, where $\rho_h$ represents the mass density of the fine mesh DG solution and $\widehat{\rho}_h$ represents mass density of the mesh-based filtering DG solution.
The left column plots these quantities for $R = \Omega$, while the right column is for $R = \Omega_1$.
Recall $\Omega$ is the whole domain, while $\Omega_1$ is the non-merging region.
For $R = \Omega$, as $N_{\theta}$ increases, the error for both the fine mesh and mesh-based filtering solutions decreases.  
In this case, the error for the mesh-based filtering solution does not decrease as rapidly as the fine-mesh errors, resulting in a visual gap between the two errors.  
However the size of the mesh-filtering error in relation to the fine mesh solution is acceptable.  
At the coarsest resolution, the $L^{\infty}$ error is on the order of $10^{-2}$ or less for both solutions.  
Across all meshes, the $L^2$ error remains below $\sim10^{-3}$, and the $L^1$ error is $\sim10^{-4}$ or less for both solutions.
The mesh-filtering error can be improved if cells are not merged maximally.  
The bottom-left panel of Figure~\ref{fig:linearRiemannX2DependenceError} depicts the error on the alternate $128\times 64$ mesh.  
With this alternate mesh, the error over the entire domain is in better agreement with the fine-mesh error.  
To further demonstrate that the errors in the left column of Figure~\ref{fig:linearRiemannX2DependenceError} arise from the merging region, $\Omega_2$, in the right column we plot the corresponding errors versus time in the non-merging region, $\Omega_1$.  
We no longer see as noticeable deviations in the $L^2$ error on the $128\times 32$ and $128\times 64$ meshes when we restrict to $\Omega_1$.

The mesh-based filter trades angular resolution for larger time steps, with a corresponding potential increase in error.  
Figure~\ref{fig:RiemannCycles} illustrates its effect on the allowable time step for the two-dimensional Riemann problem.  
The left panel shows that, for fixed $N_{r}=128$, increasing $N_{\theta}$ from 16 to 64 using the merging criterion in Equation~\eqref{eq:theta_merge_criteria} does not increase the total number of time steps.  
For $N_{\theta}=64$, the filtered solution requires $\sim3.6\times10^{3}$ time steps, compared with $\sim6.3\times10^{4}$ for the fine-mesh solution.  
With the relaxed merging criterion, the $128\times 64$ mesh requires more time steps ($\sim10^{4}$) than with the standard criterion, but substantially fewer than the corresponding fine mesh.  
The right panel of Figure~\ref{fig:RiemannCycles} shows the time step as a function of time.  
With the mesh-based filter, the time step remains independent of $N_{\theta}$ for all times, whereas it is reduced by approximately a factor of two when the angular resolution of the fine mesh is doubled.  
The abrupt decrease near $t=1.3$ results from a corresponding increase in the sound speed, while the decrease at $t=2.5$ enforces the prescribed end time.  
Thus, the time step benefit of the mesh-based filter increases with angular resolution.  
Even with relaxed merging, the allowable time step remains larger than that on the coarsest fine mesh.  
Relaxing the merging criterion reduces the error at the cost of smaller time steps.  
This tradeoff is problem dependent and must be evaluated on a case-by-case basis.  
For the solutions considered here, angular variations near the origin are modest, and the numerical results suggest that the substantial increase in time step outweighs the modest increase in error.  

\begin{figure}
    \centering
    % \begin{subfigure}[b]{0.49\linewidth}
    %     \centering
    %     \includegraphics[width=\textwidth]{RS_128x16_error.pdf}        
    % \end{subfigure}
    % \hfill
    % \begin{subfigure}[b]{0.49\linewidth}
    %     \centering
    %     \includegraphics[width=\textwidth]{RS_128x32_error.pdf}        
    % \end{subfigure}
    % \\
    % \begin{subfigure}[b]{0.49\linewidth}
    %     \centering
    %     \includegraphics[width=\textwidth]{RS_128x64_error.pdf}        
    % \end{subfigure}
    % \hfill
    % \begin{subfigure}[b]{0.49\linewidth}
    %     \centering
    %     \includegraphics[width=\textwidth]{RS_128x64_altGrid_error.pdf}        
    % \end{subfigure}
    \begin{subfigure}[b]{0.49\linewidth}
        \centering
        \includegraphics[width=\textwidth]{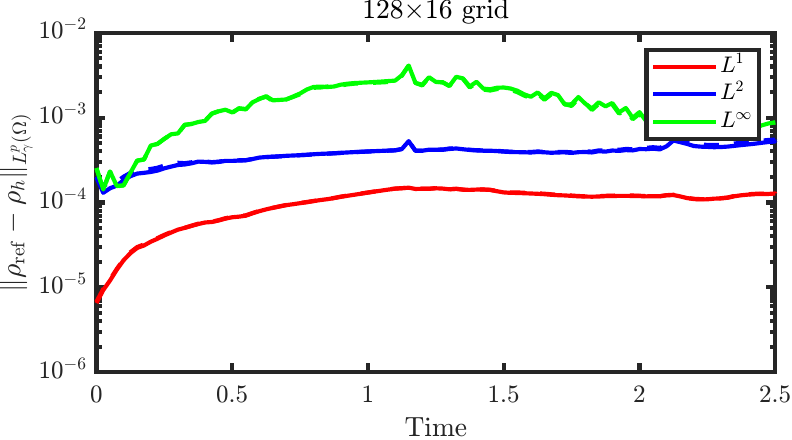}        
    \end{subfigure}
    \hfill
    \begin{subfigure}[b]{0.49\linewidth}
        \centering
        \includegraphics[width=\textwidth]{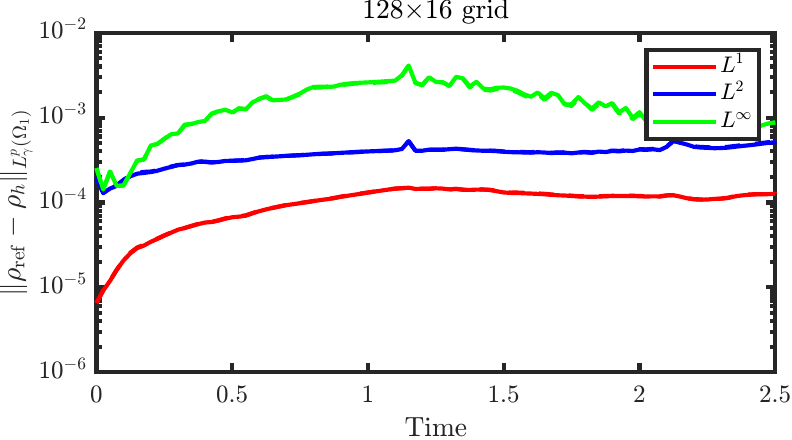}        
    \end{subfigure}
    \\
    \begin{subfigure}[b]{0.49\linewidth}
        \centering
        \includegraphics[width=\textwidth]{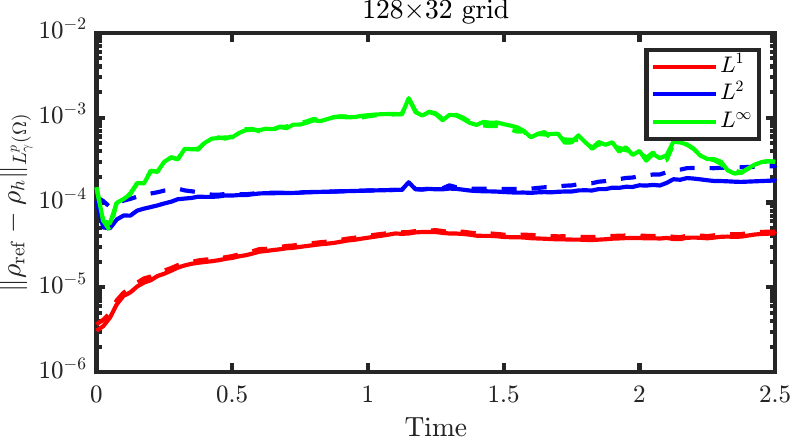}        
    \end{subfigure}
    \hfill
    \begin{subfigure}[b]{0.49\linewidth}
        \centering
        \includegraphics[width=\textwidth]{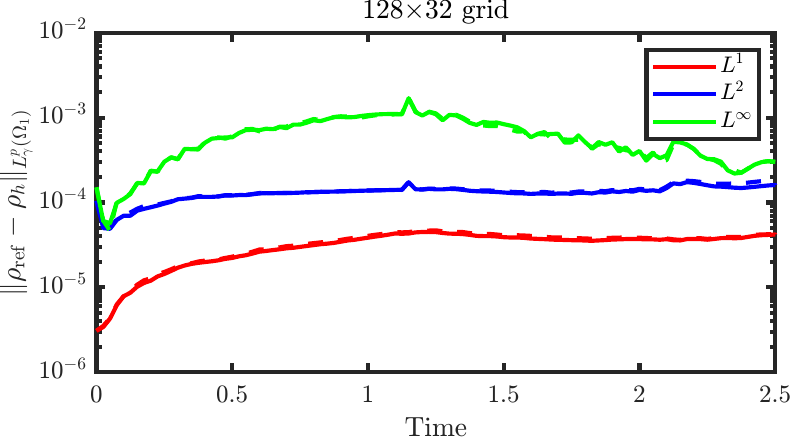}        
    \end{subfigure}
    \\
    \begin{subfigure}[b]{0.49\linewidth}
        \centering
        \includegraphics[width=\textwidth]{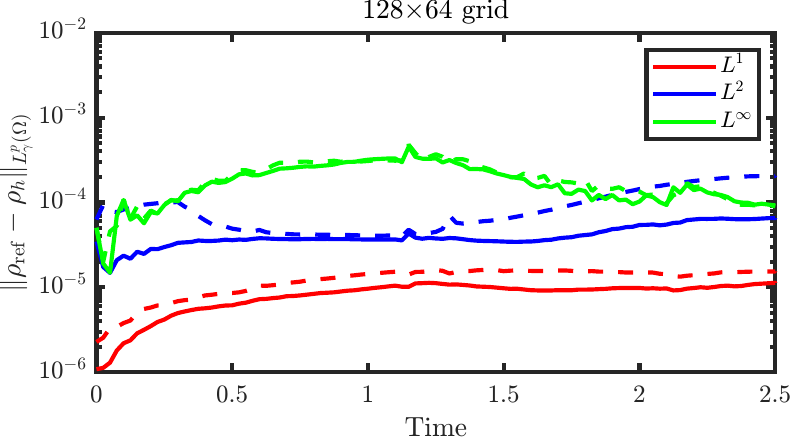}        
    \end{subfigure}
    \hfill
    \begin{subfigure}[b]{0.49\linewidth}
        \centering
        \includegraphics[width=\textwidth]{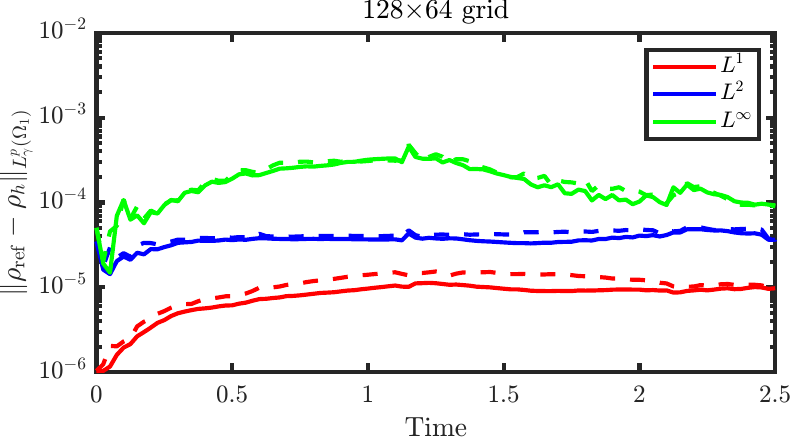}        
    \end{subfigure}
    \\
    \begin{subfigure}[b]{0.49\linewidth}
        \centering
        \includegraphics[width=\textwidth]{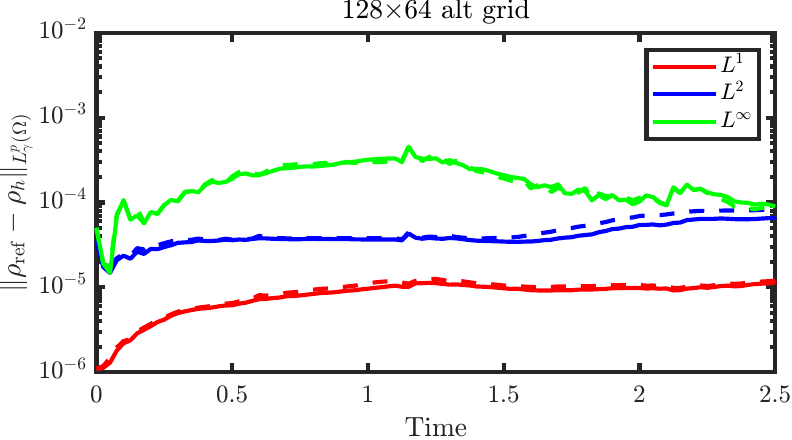}        
    \end{subfigure}
    \hfill
    \begin{subfigure}[b]{0.49\linewidth}
        \centering
        \includegraphics[width=\textwidth]{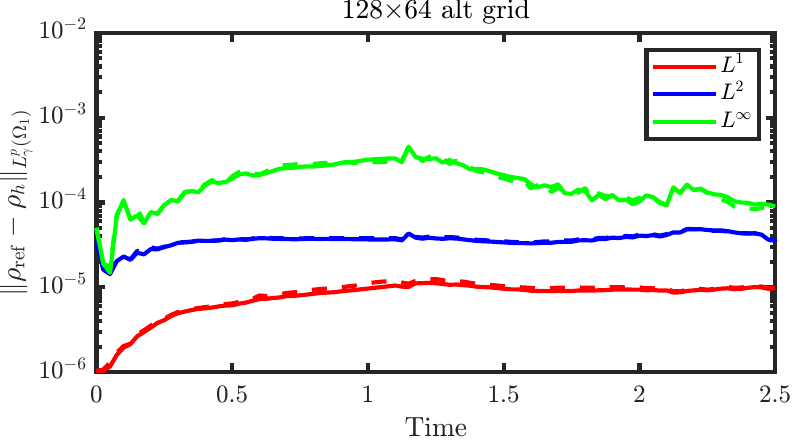}        
    \end{subfigure}
    \caption{Weighted $L^p$ error over time for the two-dimensional Riemann problem for fine-mesh (solid lines) and mesh-based filtering solutions (dashed lines) over time on $\Omega$ (left column) and $\Omega_1$ (right column) for all four meshes.
    Recall $\Omega_1$ is the union of all elements of $\widehat{\cT}$ that coincide with single elements of $\cT$
    The errors are calculated using a reference solution generated on a $128\times 128$ mesh.
    The bottom row is the result for the relaxed merging mesh.}
    \label{fig:linearRiemannX2DependenceError}
\end{figure}

% \begin{figure}
%     \centering
%     \begin{subfigure}[b]{0.49\linewidth}
%         \centering
%         \includegraphics[width=\textwidth]{RS_128x16_NMerror.pdf}        
%     \end{subfigure}
%     \hfill
%     \begin{subfigure}[b]{0.49\linewidth}
%         \centering
%         \includegraphics[width=\textwidth]{RS_128x32_NMerror.pdf}        
%     \end{subfigure}
%     \\
%     \begin{subfigure}[b]{0.49\linewidth}
%         \centering
%         \includegraphics[width=\textwidth]{RS_128x64_NMerror.pdf}        
%     \end{subfigure}
%     \hfill
%     \begin{subfigure}[b]{0.49\linewidth}
%         \centering
%         \includegraphics[width=\textwidth]{RS_128x64_altGrid_NMerror.pdf}        
%     \end{subfigure}
%     \caption{Weighted $L^p$ error over time for the two-dimensional Riemann problem for fine-mesh (solid lines) and mesh-based filtering solutions (dashed lines) over time on $\Omega_1$ for all four meshes.
%     The errors are calculated using a reference solution generated on a $128\times 128$ mesh.
%     The bottom right panel is the result for the relaxed merging mesh.}
%     \label{fig:linearRiemannX2DependenceNonmergingError}
% \end{figure}

\begin{figure}
    \centering
    \begin{subfigure}[b]{0.49\linewidth}
        \centering
        \includegraphics[width=\textwidth]{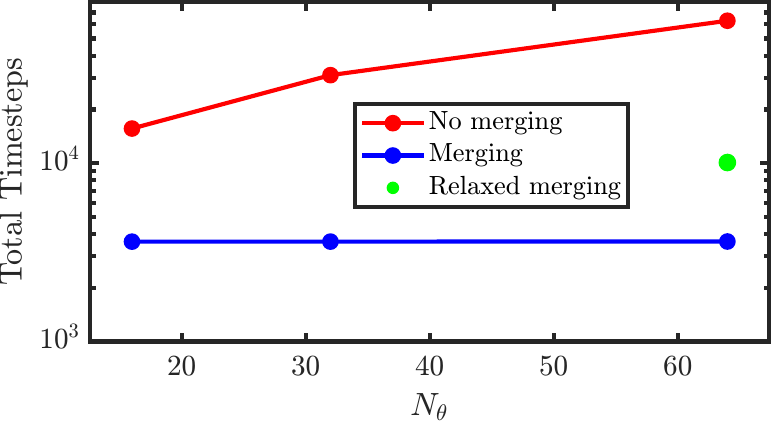}        
    \end{subfigure}
    \hfill
    \begin{subfigure}[b]{0.49\linewidth}
        \centering
        \includegraphics[width=\textwidth]{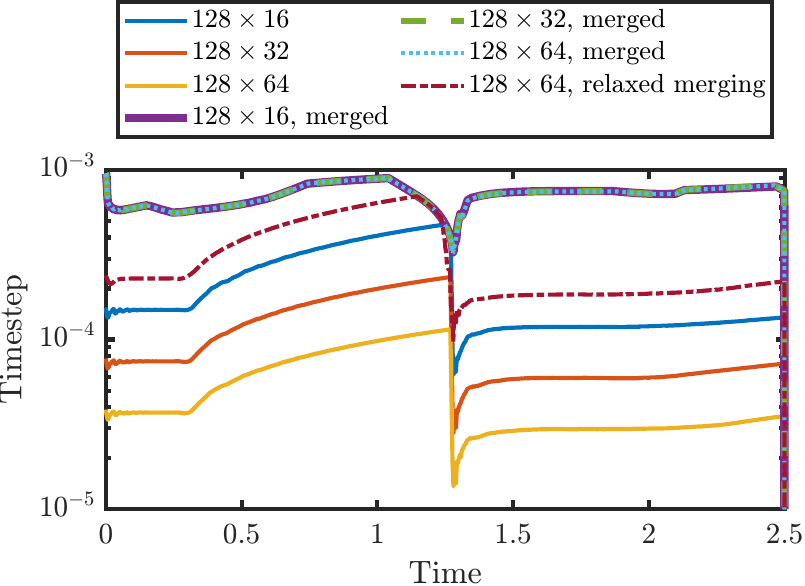}        
    \end{subfigure}
    \caption{Results for the two-dimensional Riemann problem. 
    The left panel plots the total number of time steps for $128\times N_\theta$ meshes.
    The blue line is for mesh-filtering solutions, while the red line is for fine mesh solutions.
    The green dots are for the $128\times64$ mesh with relaxed filtering.
    The right panel plots the size of the timestep versus time for the same meshes.}
    \label{fig:RiemannCycles}
\end{figure}

\subsection{Two-Dimensional Shock--Bubble Interaction}

The shock--bubble interaction test is another modification of Sod's shock tube problem.
Let $B_{0.25}(1,\pi/2)$ represent the ball of radius $0.25$ centered at the point $(r,\theta)=(1,\pi/2)$.  
The modifications include a low-density bubble inside $B_{0.25}(1,\pi/2)$, and a high-density region for $r\leq 0.1$. 
The initial conditions are given by
\begin{align}
    \rho(r,\theta) &= \begin{cases}
        100 \quad &r\leq0.1,\\
        1 \quad &0.1<r\leq 0.4,\\
        0.0125 \quad &(r,\theta)\in B_{0.25}((1,\pi/2)),\\
        0.125 \quad &\text{else,}
    \end{cases}\qquad 
    v_r(r,\theta)  = 
    v_\theta(r,\theta) = 0,\qquad
    p(r,\theta) = \begin{cases}
        1 \quad &r\leq 0.4,\\
        0.1 \quad &r>0.4.
    \end{cases}
\end{align}
We impose reflecting boundary conditions at the inner and outer boundaries for both dimensions.
We test the problem on meshes of increasing resolution using linear elements and SSPRK3 time stepping.
Specifically, the mesh resolutions tested are $64\times 64$, $128\times 128$, and $256\times 256$, however we only show results from the $256\times 256$ simulation.

\begin{table}
    \centering
    \begin{tabular}{lccccc}
        \hline
         &Min &Median &Mean &Standard Deviation &Max  \\
         \hline
         $t = 0.0$ &0 &0 &4.6495e-17 &1.1194e-16 &1.9895e-15\\
         \hline
         $t = 0.5$ &0 &1.3467e-06 &6.9803e-05 &1.0955e-03 &8.4928e-02\\
         \hline
         $t = 1.0$ &0 &1.5082e-05 &1.9777e-04 &1.3748e-03 &8.2606e-02\\
         \hline
         $t = 1.5$ &8.2473e-11 &2.9628e-05 &6.5103e-04 &2.972e-03 &1.1701e-01\\
         \hline
         $t = 2.0$ &7.1964e-11 &6.4688e-05 &1.2347e-03 &7.3922e-03 &3.2034e-01\\
         \hline
         $t = 2.5$ &5.8055e-10 &1.6176e-04 &2.9833e-03 &1.9488e-02 &1.264e\texttt{+}00\\
         \hline
    \end{tabular}
    \caption{
    Relative difference in mass density between the fine-mesh and mesh-based filter solutions at various times for the shock-bubble interaction problem.
    The min, median, mean, standard deviation, and max are reported for $\frac{\rho-\widehat{\rho}}{\rho}$, where $\widehat{\rho}$ is the density of the mesh-based filter solution.
    }
    \label{tab:BubbleDensityComparison}
\end{table}

Figure \ref{fig:BubbleTestDensity} compares the fine-mesh and filtered solutions on the $256\times 256$ mesh at $t\in\{0,0.5,1,1.5,2,2.5\}$.
This higher-resolution simulation resolves details of the bubble dynamics that are not captured in the lower resolutions.
We plot the density-gradient magnitude normalized by the density to produce a Schlieren-like visualization.  
The contour scale is logarithmic and the color map adjusted separately in each panel to emphasize the flow features.
The left half, $x(r,\theta)<0$, of individual plots is the mesh-filtering solutions, while the right half, $x(r,\theta)>0$, are the fine-mesh solutions.
Although there is no apparent visual difference between the fine-mesh and filtered solutions, Table~\ref{tab:BubbleDensityComparison} quantifies the relative difference in mass density between the two simulations.  
The median values remain below $2\times 10^{-4}$, the mean values below $3\times 10^{-3}$, and standard deviations below $2\times 10^{-2}$.  
However, these quantities increase over time.  
This increase may be due in part to the accumulation of time-integration error, since the filtered solution uses larger time steps than the fine-mesh solution.  
The largest relative differences occur along shocks and the boundary of the bubble.  
These differences are not unexpected, since a small displacement of a shock front or material boundary can produce a large pointwise relative difference, particularly where the density varies by an order of magnitude.

The benefit of achieving larger time steps with mesh-based filtering increases with resolution.
The initial fine-mesh time steps for the $64\times 64$, $128\times 128$, and $256\times 256$ meshes are $2.28\times10^{-4}$, $5.70\times10^{-5}$, and $1.43\times10^{-5}$, respectively, while the corresponding filtered time steps are $7.58\times10^{-4}$, $3.82\times10^{-4}$, and $1.92\times10^{-4}$.
The initial fine-mesh time steps decrease in size by a factor of four as the number of elements in both the $r$- and $\theta$-dimensions are doubled, while the initial filtered time steps decrease in size by a factor of two.
The filtered time steps decreasing by a factor of two, as opposed to a factor of four, indicates that the filtered time step is governed primarily by the radial resolution.
These merging time steps are approximately three, seven, and thirteen times larger than their fine-mesh counterparts.  
Consequently, the $256\times 256$ filtered simulation requires about $1.50\times10^{4}$ time steps, compared with $1.33\times10^{5}$ for the corresponding fine-mesh simulation.  

\begin{figure}
    \centering
    \begin{subfigure}[b]{0.49\linewidth}
        \centering
        \includegraphics[width=\linewidth,trim=10 50 30 60,clip]{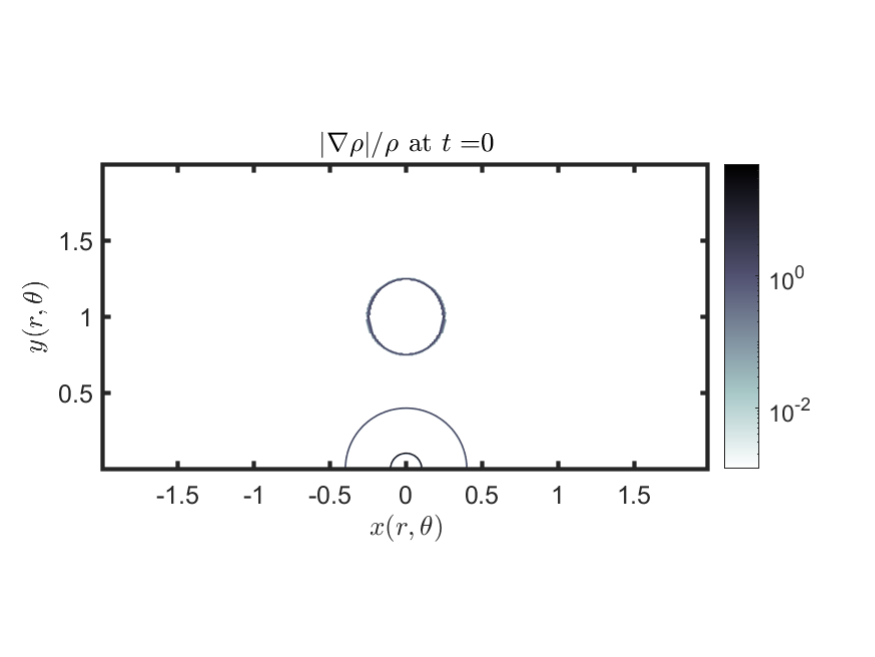}        
    \end{subfigure}
    \begin{subfigure}[b]{0.49\linewidth}
        \centering
        \includegraphics[width=\linewidth,trim=10 50 30 60,clip]{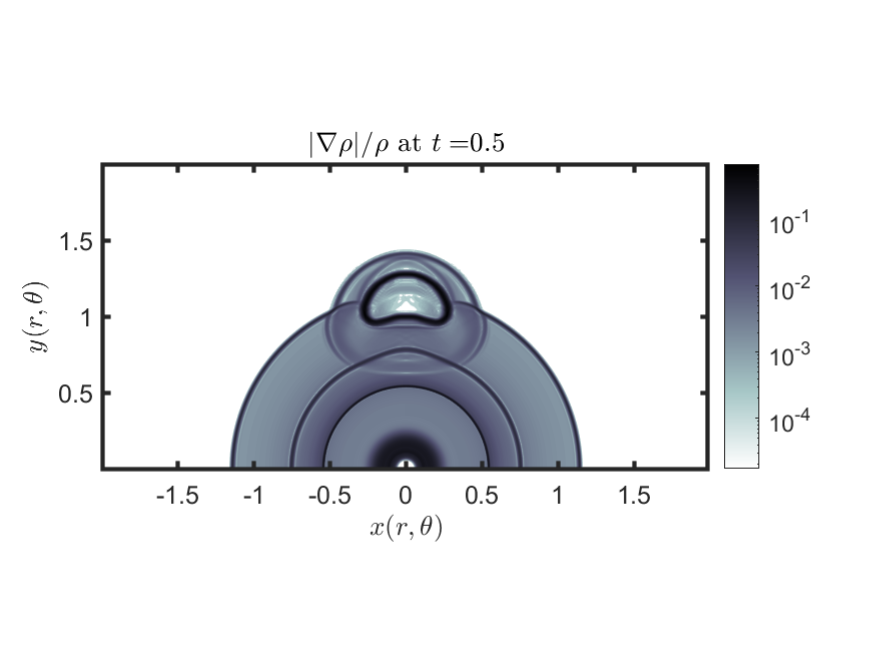}        
    \end{subfigure}
    \\
    \begin{subfigure}[b]{0.49\linewidth}
        \centering
        \includegraphics[width=\linewidth,trim=10 50 30 60,clip]{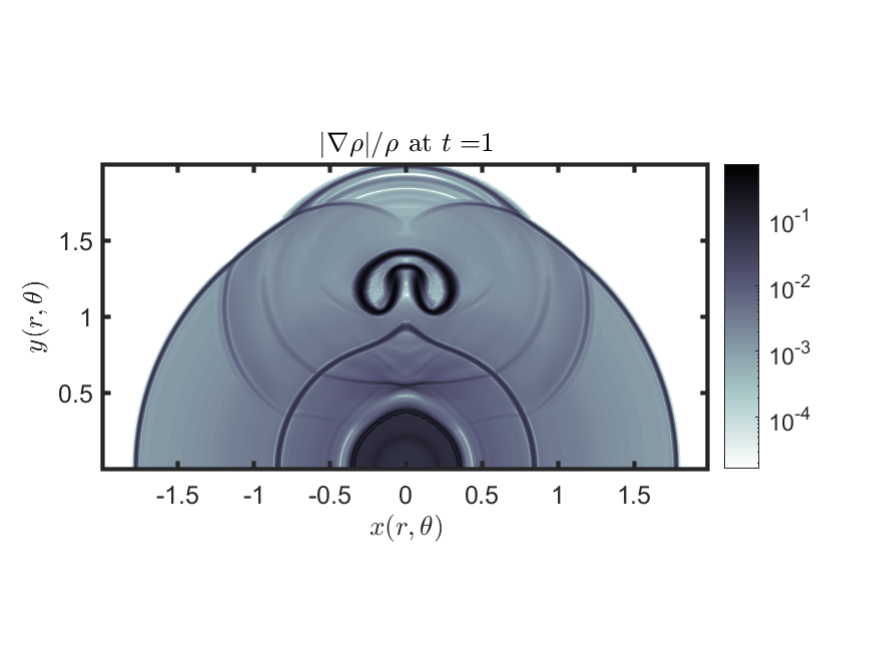}        
    \end{subfigure}
    \begin{subfigure}[b]{0.49\linewidth}
        \centering
        \includegraphics[width=\linewidth,trim=10 50 30 60,clip]{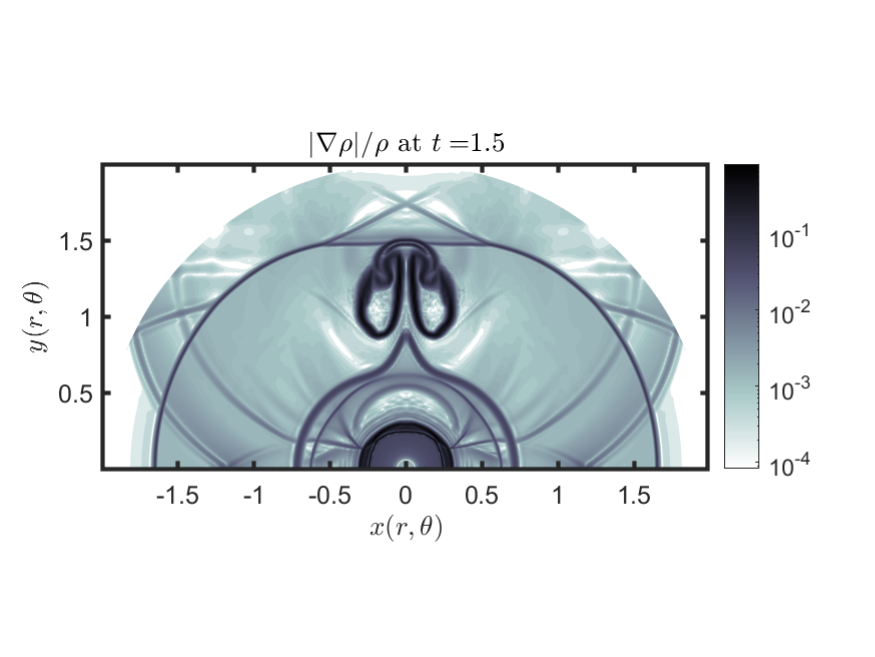}        
    \end{subfigure}
    \\
    \begin{subfigure}[b]{0.49\linewidth}
        \centering
        \includegraphics[width=\linewidth,trim=10 50 30 60,clip]{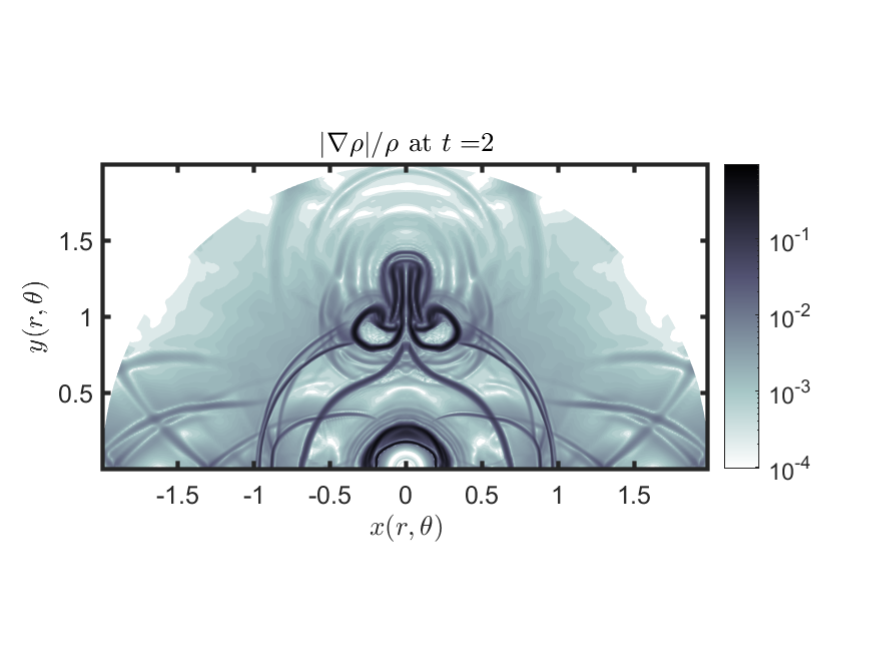}        
    \end{subfigure}
    \begin{subfigure}[b]{0.49\linewidth}
        \centering
        \includegraphics[width=\linewidth,trim=10 50 30 60,clip]{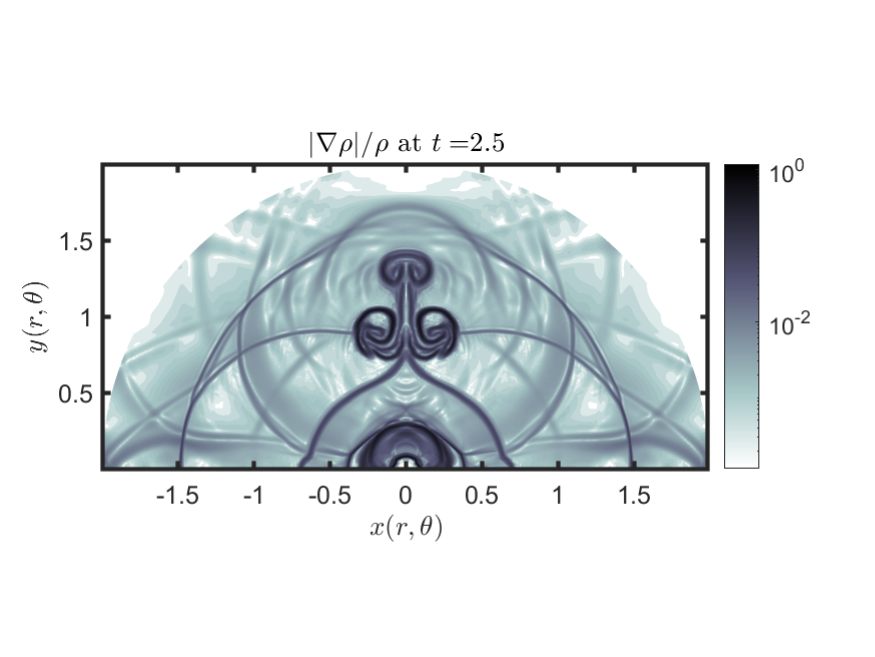}        
    \end{subfigure}
    \begin{subfigure}[b]{\linewidth}
        \centering
        \includegraphics[width=0.5\linewidth]{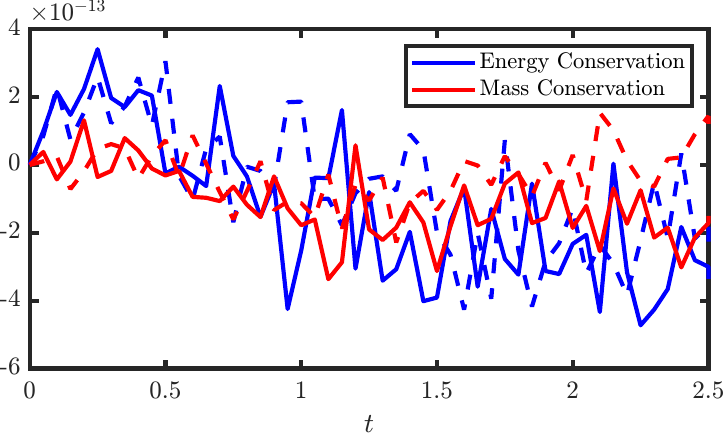}        
    \end{subfigure}
    \caption{Time series Schlieren plots for the shock--bubble interaction problem on a $256\times 256$ mesh. 
    The left half of each panel shows the result with mesh-based filtering, while the right half is the fine mesh result.  
    The plotted quantity is $|\nabla\rho|/\rho$, which provides a visualization of the shocks and wave features that develop in the solution.
    The bottom panel plots the relative change in total mass (red) and energy (blue) versus time for the shock--bubble interaction test.
    Solid lines are for fine mesh solution, while dashed lines are for mesh-based filter solution.}
    \label{fig:BubbleTestDensity}
\end{figure}

The mesh-based filter preserves the conservation properties of the RKDG method.  
The bottom panel of Figure~\ref{fig:BubbleTestDensity} compares the relative changes in mass and energy for the filtered and fine-mesh solutions on the $256\times 256$ mesh for the shock--bubble interaction test.  
For both solutions, the relative changes remain at $\cO(10^{-13})$.  
Although we present conservation results only for this test, we observe similar behavior in all tests in this section.  

% \begin{figure}
%     \centering
%     \includegraphics[width=0.5\linewidth]{Bubble_MassConservation_256x256.pdf}
%     \caption{Relative change in total mass (red) and energy (blue) versus time for the shock--bubble interaction test.
%     Solid lines are for fine mesh solution, while dashed lines are for mesh-based filter solution.}
%     \label{fig:BubbleConservation}
% \end{figure}

\subsection{Three-Dimensional Sedov--Taylor Blast Wave}

Our final three-dimensional test is the Sedov--Taylor blast wave, which models the instantaneous release of a large amount of energy within a small volume, resulting in a strong shock wave with a well-known self-similar solution \citep[e.g.,][]{landauLifshitz_1987}.  
The spherically symmetric initial conditions are given by
\begin{align}
    \rho(r) &= 1, \qquad
    v_r = 0, \qquad
    p(r) =
    \begin{cases}
        0.4\times\frac{3E_0}{4\pi R^3_0} \quad &r<R_0,\\
        0.4\times 10^{-5} \quad &r\geq R_0.
    \end{cases}
\end{align}
Here $E_0$ is the amount of energy released and $R_0$ is the radius of the sphere in which it is released.
We set $E_0 = 1$ and choose $R_0$ such that the energy is released within the first radial cell.  
We extend this initial condition to three dimensions and perform simulation on several meshes using linear elements and the SSPRK3 time integrator.  
The computational domain covers one octant of the sphere.  
We impose reflecting boundary conditions at the inner $r$ and $\theta$ boundaries, zero-gradient boundary conditions at the outer $r$ and $\theta$ boundaries, and periodic boundary conditions in $\phi$.  

The top row of Figure~\ref{fig:Sedov3DDensityEnergy} shows that the three-dimensional filtered solutions capture the blast-wave dynamics when compared to the corresponding one-dimensional reference solution and remain spherically symmetric.
We have investigated how well spherical symmetry is maintained in these simulations by computing the standard deviation of $\rho$ in the angular dimensions for each radial node.  
We computed these values at time intervals of $0.05$.  
For all the filtered solutions the standard deviation remains below $10^{-12}$.  
The exception is the $64\times 16\times 16$ filtered solution at $t=1$, where the standard deviation remains below $10^{-6}$.
However, this enhanced standard deviation occurs between $r=0.7$ and $r=1$.  
Outside this region, the standard deviation is below $10^{-12}$.  
We note these values are similar and often times smaller than the respective values for the fine-mesh solutions.

The principal benefit of the filter is the increase in the allowable time step, which is greater in three dimensions, as shown in the bottom row of Figure~\ref{fig:Sedov3DDensityEnergy}.  
For the $64\times N_\theta\times N_\phi$ meshes, the filtered time steps are essentially independent of $N_\theta$ and $N_\phi$, whereas the fine-mesh time steps decrease with increasing angular resolution.  
The filtered time step is approximately 25 times larger than the fine-mesh time step on the $64\times 4\times 4$ mesh and 100 times larger on the $64\times 8\times 8$ mesh.  
Extrapolating this scaling to the $64\times 16\times 16$ mesh suggests an increase of approximately a factor of 400.  
We did not perform the corresponding fine-mesh simulation because of the small allowable time step.  

\begin{figure}
    \centering
    \begin{subfigure}[b]{0.49\linewidth}
        \centering
        \includegraphics[width=\textwidth]{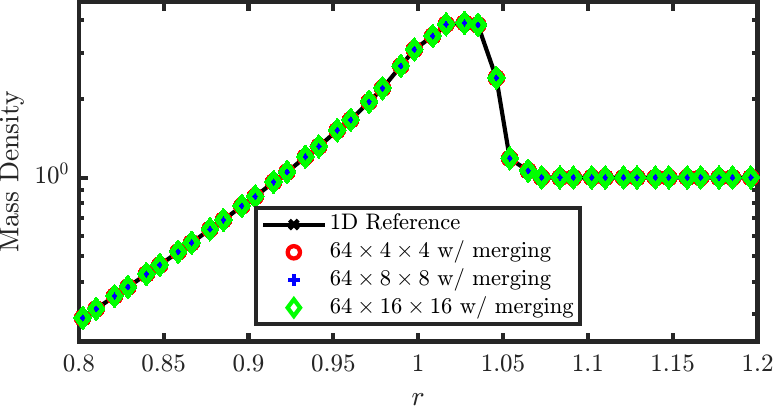}        
    \end{subfigure}
    \hfill
    \begin{subfigure}[b]{0.49\linewidth}
        \centering
        \includegraphics[width=\textwidth]{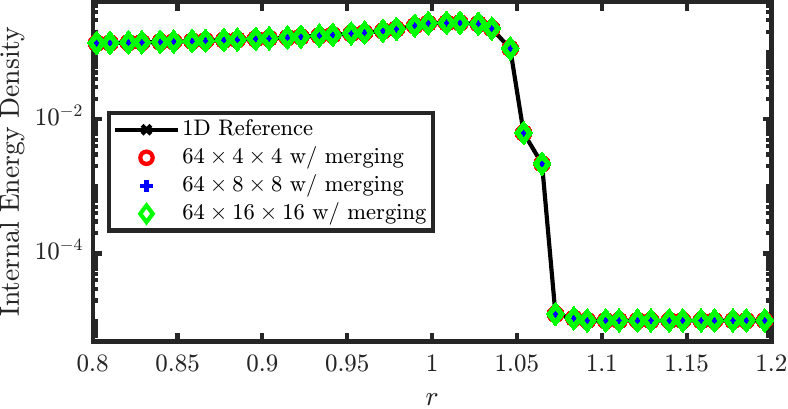}        
    \end{subfigure}
    \begin{subfigure}[b]{0.49\linewidth}
        \centering
        \includegraphics[width=\textwidth]{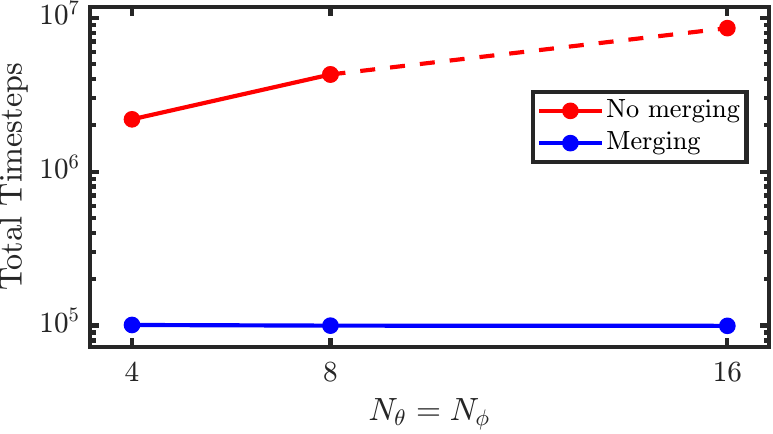}        
    \end{subfigure}
    \hfill
    \begin{subfigure}[b]{0.49\linewidth}
        \centering
        \includegraphics[width=\textwidth]{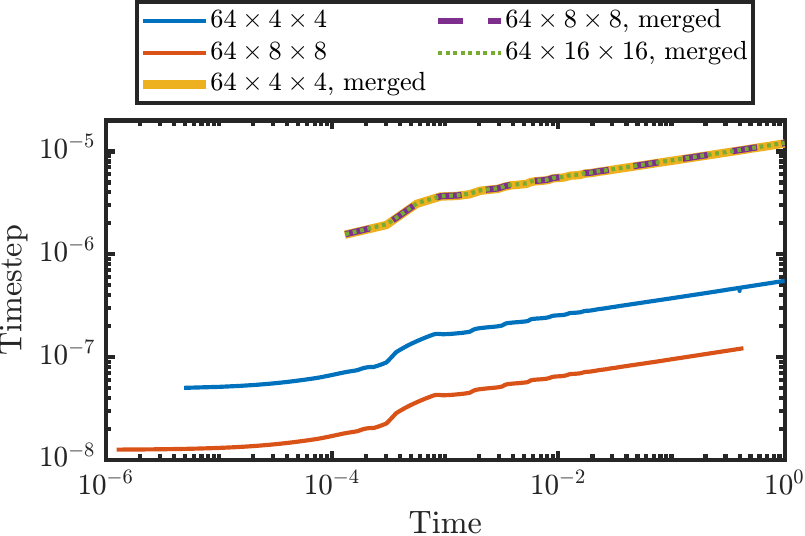}        
    \end{subfigure}
    \caption{Top row: radial profiles of the mass density (left) and internal energy density (right) for the 3D Sedov--Taylor blast wave at time $t=1$.
    The radial extent of the plots is restricted to a region around the wavefront.
    The solid black line is a 1D reference solution with $N_{r}=64$.  
    The red, blue, and green markers represent the mesh-based filtering results obtained on the $64\times 4\times 4$, $64\times 8\times 8$, and $64\times 16\times 16$ meshes, respectively.
    Bottom row: the total number of time steps (left) and the time step versus time (right) for the different $64\times N_\theta\times N_\phi$ mesh sizes used for the 3D Sedov--Taylor blast wave.
    In these tests, $N_\theta = N_\phi$.
    The red dashed line in the left panel is an estimate of the total number of time steps for a $64\times 16\times 16$ fine-mesh simulation.
    }
    \label{fig:Sedov3DDensityEnergy}
\end{figure}

\section{Summary, Conclusions, and Outlook}
\label{sec:Conclusion}

We have proposed a mesh-based filtering strategy for explicit discontinuous Galerkin methods that mitigates the severe geometric CFL restrictions induced by converging coordinate lines in curvilinear coordinate systems such as spherical-polar coordinates.  
The key idea is to construct a \emph{merged mesh} from an underlying structured \emph{fine mesh} by merging neighboring fine-mesh elements, thereby reducing extreme cell anisotropies and yielding a significantly larger stable time step.  
The merged mesh is constructed once at initialization.  
During time evolution, the DG solution is projected onto the merged mesh and prolonged back to the fine mesh.  
This mesh-based filter, which can be incorporated into existing structured-mesh DG solvers, enables time integration with the larger stable time step associated with the merged mesh.

In the one-dimensional setting, we have proven that applying the mesh-based filter is equivalent to evolving the RKDG method directly on the merged mesh and verified this result using numerical examples for the linear transport and nonlinear Burgers' equations.  
Consequently, the filtered method inherits the stability properties of the RKDG method on the merged mesh.  
We have developed merging strategies for spherical-polar meshes in two and three spatial dimensions, implemented them in the DG-based toolkit for high-order neutrino radiation hydrodynamics (\texttt{thornado}), and demonstrated the effectiveness of mesh-based filtering for the Euler equations using a Riemann problem with induced angular dependence, a shock--bubble interaction problem, and the Sedov--Taylor blast wave problem.  
With a fixed number of radial elements, mesh-based filtering permits refinement in the angular dimensions without reducing the stable time step, enabling significantly faster high-angular-resolution simulations.  

The larger stable time steps enabled by the mesh-based filtering come at the cost of some local loss of resolution in regions where elements are merged.  
As the angular resolution of the fine mesh increases, merged elements in the filtering region represent increasingly larger collections of fine-mesh elements.  
Nevertheless, our numerical experiments suggest that the resulting loss of accuracy is modest and is outweighed by the computational savings afforded by the larger time step.  
For problems without thin boundary layers, we believe that constructing merged meshes with approximately uniform proper-distance aspect ratios provides a reasonable balance between accuracy and efficiency.  
Applications requiring enhanced angular resolution can be accommodated through less aggressive merging, albeit with a corresponding reduction in the stable time step.  
More generally, the mesh-based filtering strategy is intended to alleviate geometric anisotropies introduced by the coordinate system; when large anisotropies are required by the underlying physics, alternative approaches such as implicit time integration may be more appropriate.  

The present work establishes a foundation for several future developments of mesh-based filtering for DG methods.  
An important next step is to investigate the scalability of the approach in distributed-memory environments, where mesh merging may introduce load imbalance, increased communication, and constraints on parallel domain decomposition \citep[see, e.g.,][]{ji2023,cardall_etal_2026}.  
Another possible direction is the development of physical-constraint-preserving formulations that maintain positivity and other admissibility properties in curvilinear coordinates with strong geometric anisotropies, particularly in combination with mesh-based filtering \citep[e.g.,][]{zhangShu_2011,wu2017,WX2021}.
Finally, it would be interesting to extend the approach to implicit-explicit (IMEX) time integrators \cite[e.g.,][]{ascher_etal_1997,pareschiRusso_2005}, where the explicit stages benefit from the larger stable time step provided by the merged mesh, while computationally intensive local implicit solves could be performed on the merged mesh to reduce cost, e.g., in the context of neutrino kinetics in core-collapse supernova simulations \citep{mezzacappa_etal_2020}.

% ========================

% End of main body of text

% ========================

\begin{acknowledgments}
E.~Endeve acknowledges support from the National Science Foundation's (NSF's) Gravitational Physics program under Grant No.~2409148, NSF's Computational Mathematics program under Grant No.~2309591, and NSF's Cyberinfrastructure for Sustained Scientific Innovation program under Grant No.~2513245. 
Y.~Xing acknowledges support from the NSF's Computational Mathematics program under Grant No.~2309590.
\end{acknowledgments}

\software{
\texttt{thornado} (\href{https://github.com/endeve/thornado/tree/CellMerging}{Mesh-based filtering branch})
          }

\bibliography{references}{}
\bibliographystyle{aasjournalv7}

\end{document}